\documentclass[12pt,reqno,final]{amsart}
\usepackage{fullpage}
\usepackage{amsthm}
\usepackage{mathtools}
\usepackage{amsfonts}
\usepackage{amsmath,amssymb}
\numberwithin{equation}{section}
\usepackage{tikz-cd}
\usepackage{bbm}
\usepackage{comment}
\allowdisplaybreaks
\usepackage[margin=0.7in]{geometry}
\usepackage{imakeidx}
\usepackage[normalem]{ulem}
\makeindex
\usepackage{nomencl}
\makenomenclature

\renewcommand{\nompreamble}{\noindent The following list describes several conventions (symbols, notations, etc.) that are used in the article, along with the page number of their debut.}

\usepackage{enumitem,kantlipsum}
\RequirePackage[colorlinks,citecolor=blue,urlcolor=blue,backref=page]{hyperref}

\usepackage{tikz}
\usepackage{todonotes}
\allowdisplaybreaks

\usetikzlibrary{shapes,positioning,intersections,quotes}

\usepackage{datetime2}

\usepackage[marginal]{showlabels}

\makeatletter
\newtheorem*{rep@theorem}{\rep@title}
\newcommand{\newreptheorem}[2]{
\newenvironment{rep#1}[1]{
 \def\rep@title{#2 \ref{##1}}
 \begin{rep@theorem}}
 {\end{rep@theorem}}}
\makeatother

\newcommand{\be}{\begin{equation}}
\newcommand{\ee}{\end{equation}}
\newcommand{\bea}{\begin{eqnarray}}

\newcommand{\eea}{\end{eqnarray}}

\newcommand\bmodif{\begin{modif}}
	\newcommand\emodif{\end{modif}}

\newcommand{\E}{\mathbb{E} }

\newcommand{\remove}[1]{}

\newtheorem{theorem}{Theorem}[section]

\newtheorem{lemma}[theorem]{Lemma}
\newtheorem{prop}[theorem]{Proposition}

\newtheorem{definition}[theorem]{Definition}

\newtheorem{example}[theorem]{ Example}

\theoremstyle{definition}
\newtheorem{remark}{Remark}

\numberwithin{equation}{section}

\numberwithin{equation}{section}

\def\0{{\bf 0}}

\newcommand{\M}{\mathbb {M} }

\def\R{\mathbb{R}}

\def\bBY{\mathbb{Y}}

\newcommand{\1}[1]{\mathbf{1}(#1)}

\newcommand{\red}[1]{\textcolor{red}{#1}}

\providecommand{\keywords}[1]

\title{Existence and Convergence of Interacting Particle Systems on Graphs}
\author{Kuldeep Guha Mazumder$\;^*$}
\thanks{$^*$ Theoretical Statistics and Mathematics Unit, Indian Statistical Institute, Bangalore. \textit{rs\_math1903@isibang.ac.in}}

\begin{document}
\maketitle
\begin{abstract}
\noindent We give a general existence and convergence result for interacting particle systems on locally finite graphs with possibly unbounded degrees or jump rates. We allow the local state space to be Polish, and the jumps at a site to affect the states of its neighbours. The two common assumptions on interacting particle systems are uniform bounds on degrees and jump rates. In this paper, we relax these assumptions and allow for vertices with high degrees or rapid jumps. We introduce new assumptions ensuring that such vertices are placed sufficiently apart from each other and hence the process does not blow up. Our assumptions involve finitude of certain weighted connective constants on the square graph of the underlying graph and our proofs proceed by showing that these assumptions imply non-percolation of the Poisson graphical construction. For some random graph models, we give practically verifiable sufficient conditions under which our assumptions hold almost surely. These conditions involve exponential growth of a fractional moment sum of probabilities of self-avoiding walks from each vertex and that of product moments of fixed powers of jump rates. Using these conditions, we show the existence of interacting particle systems with possibly unbounded jump rates like contact processes, consensus formation models, evolutionary models, etc., on random graphs which can lack uniform bounds on degrees almost surely, e.g., long-range percolations on quasi-transitive graphs, and geometric random graphs on Delone sets.

\noindent \textbf{MSC $\mathbf{2020}$
subject classifications.} Primary: 60K35; Secondary: 60J25, 60J76.
\end{abstract}

\nomenclature{$u\sim v$}{vertices $u$ and $v$ have an edge between them, \pageref{nom:sim}}
\nomenclature{$u\not\sim v$}{vertices $u$ and $v$ do not have an edge between them, \pageref{nom:nsim}}
\nomenclature{$W<V$}{$W$ is a finite subset of $V$, \pageref{nom:finsub}}
\nomenclature{$\mathcal{N}_v$}{neighbourhood of vertex $v$, \pageref{nom:nbdv}}
\nomenclature{$\mathcal{N}_v^+$}{$2$-neighbourhood of vertex $v$, \pageref{nom:nbdv+}}
\nomenclature{$\mathcal{N}_W$}{neighbourhood of set $W$, \pageref{nom:nbdW}}
\nomenclature{$\mathcal{N}_W^+$}{$2$-neighbourhood of set $W$, \pageref{nom:nbdW+}}
\nomenclature{$\mathcal{V}^2$}{Square graph of graph $\mathcal{V}$, \pageref{nom:V+}}
\nomenclature{$\Delta_{\mathcal{V},v}$}{local connective constant of graph $\mathcal{V}$ at vertex $v$, \pageref{nom:lcc}}
\nomenclature{$\text{SAW}_\mathcal{V}(v)$}{set of all self-avoiding walks from vertex $v$ in graph $\mathcal{V}$, \pageref{nom:sawv}}
\nomenclature{$\text{SAW}_{\mathcal{V},n}(v)$}{set of all self-avoiding walks from vertex $v$ of length $n$ in graph $\mathcal{V}$, \pageref{nom:sawvn}}
\nomenclature{$\text{SAW}_{\mathcal{V},n}^*(v)$}{set of all self-avoiding walks from vertex $v$ of length $n$ in graph $\mathcal{V}^2$ which are remnants of some self-avoiding walk from $v$ in graph $\mathcal{V}$, \pageref{nom:sawnv*}}
\nomenclature{$\mathbb{Y}$}{Polish local state space, \pageref{nom:lss}}
\nomenclature{$\mathcal{C}$}{set of all cylinder functions on $\mathbb{Y}^V$, \pageref{nom:fincrd}}
\nomenclature{$\mathcal{C}(A)$}{set of all cylinder functions on $\mathbb{Y}^V$ with support $A$, \pageref{nom:Acrd}}
\nomenclature{$\Theta_{v,(n)}$}{$n$-step simple jump trail rate from vertex $v$ of graph $\mathcal{V}$, \pageref{eq:ndjrt}}
\nomenclature{$\Theta_{v,(n)}^*$}{$n$-step double jump trail rate from vertex $v$ of graph $\mathcal{V}$, \pageref{eq:ndjrt}}
\nomenclature{$\Theta_{v}^*$}{double jump trail rate from vertex $v$ of graph $\mathcal{V}$, \pageref{eq:djrt}}
\nomenclature{$\Theta_{v}$}{simple jump trail rate from vertex $v$ of graph $\mathcal{V}$, \pageref{eq:djrt}}
\nomenclature{$E_t(v,w)$}{vertex $v$ affects vertex $w$ before time $t$, \pageref{nom:etvw}}
\nomenclature{$E_t'(v,w)$}{vertex $v$ directly affects vertex $w$ before time $t$, \pageref{nom:etvw'}}

\section{Introduction}\label{sec:intro}
In this paper, we study the existence of interacting particle systems on locally finite graphs and show that the infinite-volume dynamics of such systems can be obtained as the limit of the corresponding processes defined on an increasing sequence of finite subgraphs. By an interacting particle system, we refer to a Markov process $\xi := \big((\xi_t(v))_{v \in V}\big)_{t \geq 0}$ on $\mathbb{Y}^V$ where $V$ is the vertex set of the underlying locally finite (but possibly infinite) graph $\mathcal{V}$ and $\mathbb{Y}$ (called the \emph{local state space}, i.e., the space of configurations at any vertex of $\mathcal{V}$) is a Polish space. For each $v\in V$, $(\xi_t(v))_{t\geq0}$ is a pure jump process whose jump rates and evolution of states depend on $(\xi_t(\cdot))_{t\geq0}$ at neighbouring vertices. In other words, the updates at any site affect both the local jump rates and the states of its neighbours. Three basic questions regarding such models on infinite graphs are existence, uniqueness and convergence (of the model on a finite approximating sequence of graphs to that on the infinite graph). To answer these questions, the two common assumptions made in the literature \cite{Durrett1995,liggett2012interacting,Penrose2008existence} are that $\mathcal{V}$ has uniformly bounded degrees and that the jump rates of the process are uniformly bounded over the entire graph. Intuitively, these assumptions ensure that ``badly behaving vertices" (i.e., those with high degrees or rapid jumps) are absent from the graph and thus the process does not ``{blow up}" (in a sense that will be explained in Subsection \ref{subs:graphical}). However, these assumptions fail for interacting particle systems on many graphs, including many random graphs. In this paper, we show the existence and convergence of various interacting particle systems under weaker assumptions than those mentioned above. Informally, we require that vertices of high degrees or rapid jumps are not too close to each other, thereby ensuring that the process does not {blow up}. This interplay between the jump rates and the geometry of the graph is formalized in Theorem \ref{thm:exist_cvg_ips}, wherein it is captured by the finitude of {jump trail rate}s (defined in Definition \ref{def:djrt}) from every vertex. {In Theorem \ref{thm:randips}, we give two practically verifiable sufficient conditions for interacting particle systems on random graphs with deterministic vertex sets, under which the assumptions of Theorem \ref{thm:exist_cvg_ips} hold almost surely. The first condition concerns the geometry of the random graph -- for a fixed $p>1$ and for each $v\in V$, the $\ell^{1/p}$-sum of the probabilities of self-avoiding walks of length $n$ starting from $v$ grows at most exponentially in $n$. The second condition concerns the assignment of jump rates -- the $p'$-th product moment of the jump rates at any $n$ distinct vertices grows at most exponentially in $n$, where $p'$ is the H\"older conjugate of $p$. Since $p>1$, the first condition is slightly stronger than the expected number of self-avoiding walks of length $n$ from each vertex growing exponentially in $n$.} This condition should hold for many random graphs, including the ones with almost surely finite connective constants (and hence those with almost surely uniformly bounded degrees). In Section \ref{sec:rgex} we discuss examples of two such random graph models -- long-range percolation on deterministic graphs like quasi-transitive graphs, and geometric random graphs on various well-spaced point sets in the Euclidean space, like various Delone sets. These random graphs can lack a uniform degree bound almost surely. Further, although the jump trail rates introduced in Theorem \ref{thm:exist_cvg_ips} may involve the square graph, the conditions in Theorem \ref{thm:randips} are on the original graph, which makes the latter theorem simpler in terms of applicability. We demonstrate our results by giving examples of interacting particle systems where the jump rates proportional to the vertex degrees, such as consensus formation models, contact processes, interacting urn models, certain sandpile models, and Bak-Sneppen-type evolution models; see Section \ref{sec:ipsex}.

We shall now place our results and proof techniques in the context of the existing literature. There are different models of interacting particle systems but we focus only on those on graphs not evolving in time. Some classical models of such systems are the voter model, the contact process, the exclusion process and birth-death processes. One can find discussions on these models in \cite{liggett2012interacting,durrett1980contact,Penrose2008existence,bezborodovthesis}. Some other models of interest include zero range processes \cite{evans2005}, adsorption models \cite{Evans1993,Penrose2008existence,Shcherbakov_2024}, sandpile models \cite{Dhar_1999,corry2018divisors}, Bak-Sneppen models \cite{PhysRevE.53.414,benari2017local}, etc. An exposition of the literature on interacting particle systems specifically on random graphs can be found in \cite{frank2024}.

It is commonly assumed that the local state space $\mathbb{Y}$ is finite \cite{Durrett1995,liggett2012interacting} or compact \cite{fristedtgray} and that the jumps at a site affect only the jump rates and not the states of its neighbours. Penrose \cite{Penrose2008existence} considers Polish state space and allows for jumps occurring at any given site to affect the states of its neighbours as well. However, as already mentioned, in most of the present literature where the existence-type results have been studied, including \cite{Penrose2008existence}, two assumptions seem to be at the core -- the uniform degree bound and the uniform bound on the jump rates. There are a few exceptions to this. Bezborodov \cite[Section $5.2$]{bezborodovthesis} studies existence and uniqueness of various birth-and-death process using SDE methods. The vertex set is $\mathbb{Z}^d$ and the local state space is $\mathbb{N}_0$, the set of all non-negative integers, with discrete topology. The jump rates are allowed to be unbounded, with the birth rate satisfying certain monotonicity and Lipschitz properties and the death rate satisfying a monotonicity property, all described therein. Both the birth and the death rates at any vertex are local, i.e., depend only on the vertices within a finite distance. This work is extended by Bezborodov, Kondratiev and Kutoviy \cite{Bezborodov_2019}, allowing the jump rates to be unbounded and to have long-range dependencies, with the dependence decaying sufficiently fast with the distance between the vertices. Gantert and Schmid \cite{gantert2020} use graphical construction (to be described in Subsection \ref{subs:graphical}) to show that simple exclusion processes exist on specific random graphs (augmented Galton-Watson trees) which may lack a uniform degree bound. The jump rates are uniformly bounded. The local state space is the $2$-point set $\{0,1\}$. The proof exploits specific features of the particle system. Ganguly and Ramanan \cite{ganguly2022hydrodynamic} use graphical construction-based arguments to show that a sequence of particle systems defined on a corresponding sequence of finite random graphs converge (in an appropriate sense) to a possibly non-Markovian  particle system (called the \emph{hydrodynamic limit}) on the infinite graph, given the graph sequence converges to the infinite graph in local weak sense, and the limiting infinite graph satisfies a certain ``finite dissociability" property almost surely. The class of dissociable graphs include bounded degree graphs as well as some Galton-Watson trees which may not have a uniform degree bound. The jump rate kernels satisfy certain regularity assumptions. The local state space is $\mathbb{Z}$ equipped with discrete topology. Our work is complementary to \cite{bezborodovthesis}, \cite{Bezborodov_2019}, \cite{gantert2020} and \cite{ganguly2022hydrodynamic} -- the assumptions made in our work are not subsumed in them and vice versa. We are not aware of any existence-type work on particle systems on graphs that simultaneously allows unbounded vertex degrees, unbounded jump rates and general Polish local state spaces, and to the best of our knowledge, our work is the first one in the current literature to do so.
 
We build our work mostly on the framework and the proof techniques in \cite{Penrose2008existence}, using the idea of graphical construction. \emph{Graphical construction} (to get an informal idea of the concept, see introduction of Section \ref{sec:graph}), introduced by T. E. Harris in \cite{Harris1972}, is a fundamental tool to prove the existence of interacting particle systems. Briefly, a graphical construction is an oriented space-time graph built on the underlying graph $\mathcal{V}$ using a given set of jump rate kernels, that captures the dependencies among the vertices of $\mathcal{V}$. We adapt the graphical construction used by Penrose in \cite{Penrose2008existence}, which is in turn a refinement of Harris' technique to account for the fact that the dependencies are between $2$-neighbours and not only immediate neighbours. We further refine the former by making the graphical construction ``inhomogeneous", this is to account for the fact that there is no uniform bound on the jump rates in our case. The idea of proving existence using graphical construction, in broad strokes, is that when a graphical construction does not percolate in any finite time, one can construct a family of interacting particle systems on $\mathcal{V}$. The search for sufficient conditions for the existence of interacting particle systems thus reduces to that for the non-percolation of the graphical construction in any finite time, we do this in Proposition \ref{prop:exstcggc}. We show that non-percolation implies existence of the process, and construct the process explicitly in Proposition \ref{prop:ggcexstc}. We also show that the particle system on the infinite graph may be viewed as a limit (in a sense to be specified in Theorem \ref{thm:exist_cvg_ips}) of a sequence of particle systems on a corresponding sequence of finite ‘windows’ $W_n\subset V$ such that $W_n\uparrow V$. We note that our proof is a constructive one, i.e., we explicitly construct a family of Markov processes using graphical construction. That being said, graphical construction is not the only tool used to prove existence-type results. Another commonly used technique is studying the processes as solutions to stochastic differential equations. This can be found in the more recent works that deal with specific models (e.g., birth-death-displacement models, CSA models,
etc.) like \cite{bezborodovthesis,Bezborodov_2019}, etc.
 
Our work does not study uniqueness of the particle systems. It will be of interest to have a result on uniqueness for a general model under assumptions weaker than those of uniformly bounded degrees and jump rates. Another assumption that we would like to relax is the local pure jump nature of the process. A simple extension would be to allow for deterministic evolution of the system in between the jump-type events, the deterministic evolution at one site affecting only the jump rates of the neighboring sites, e.g., \cite{sidora1995}. Another way to extend our work is to study the existence and uniqueness of particle systems which are possibly non-Markovian such as those in \cite{ganguly2022hydrodynamic}. Our work also does not cover particle systems (e.g., birth-death, displacement, migration, adsorption, deposition, etc.) in continuum (i.e., in $\mathbb{R}^d$ or in a general metric space). These processes can be seen as stochastic processes taking values in the set of all counting measures on the continuum space they live in. Such an approach is different from ours. It does not involve any graph structure, and hence an equivalent of the uniform degree bound condition in such cases is that the interactions should be uniformly bounded, i.e., each particle interacts with other particles that are within a fixed distance. There are some works on particle systems in continuum where the uniform bound assumptions on degree and / or on jump rates have been relaxed, e.g., \cite{garciakurtz2006,Bezborodov_2022, fernandez2004spatial}. Penrose \cite{Penrose2008existence} uses some discretization techniques to reduce the continuum framework of such models to a lattice framework, and this is an approach that we intend to pursue as a future extension of our current work. Our results apply to random graphs with deterministic vertex sets, and we believe that they can be extended to more general random graphs. In fact, that we can construct interacting particle systems on random graphs on various well-spaced point sets in the Euclidean space is strongly suggestive of the possibility that we can indeed construct interacting particle systems on random graphs with random vertex sets, e.g., germ-grain-type random graphs on point processes. This is something that we plan to work on in future. Another important class of models that appear difficult to be covered by our results are interacting particle systems with genuinely state-dependent jump rates, e.g., sandpile-type models where the total jump rate at a vertex is proportional to the mass at that vertex. Extending the present framework to such systems would require combining the geometric non-percolation approach developed here with quantitative control on the evolution of local states along the dynamics, so that the graphical construction can be controlled through suitably truncated, state-dependent jump rates.
 
We now describe the structure of the paper. In Section \ref{sec:prelim}, we review some concepts regarding Markov processes and graph theory. In Section \ref{sec:excon}, we describe the framework, state and discuss Theorem \ref{thm:exist_cvg_ips}, our central result on existence and convergence of interacting particle systems in a deterministic setting. Theorem \ref{thm:exist_cvg_ips} finds many applications in random graphs. Therefore, in Section \ref{sec:exconr}, we discuss Theorem \ref{thm:exist_cvg_ips_r}, a restatement of Theorem \ref{thm:exist_cvg_ips} in a random setting. The assumptions of Theorem \ref{thm:exist_cvg_ips_r} involve the square graph and are difficult to verify in examples directly, and therefore we further state and prove Theorem \ref{thm:randips}, a consequence of Theorem \ref{thm:exist_cvg_ips_r} that involves only the original graph and can be used more directly. In Sections \ref{sec:rgex} and \ref{sec:ipsex}, we respectively discuss examples of some random graphs and interacting particle systems on them satisfying conditions of Theorem \ref{thm:randips}. In Section \ref{sec:graph}, we show that the conditions of Theorem \ref{thm:exist_cvg_ips} imply non-percolation of graphical construction in time (Proposition \ref{prop:exstcggc}), and the latter implies the conclusions of Theorem \ref{thm:exist_cvg_ips} (Proposition \ref{prop:ggcexstc}) -- thereby proving Theorem \ref{thm:exist_cvg_ips} in effect. A list of conventions (symbols, notations, etc.) which are non-standard and are used in this paper have been added at the end for quick reference.

\section{Preliminaries}\label{sec:prelim}
In this section we recall some notions related to continuous-time Markov processes (Subsection \ref{subs:markov}) and also some graph theoretic notions (Subsection \ref{subs:graph}) particularly those related to self-avoiding walks. Throughout this paper, $\mathbb{R}$, $\mathbb{R}_{\geq 0}$, $\mathbb{N}$ and $\mathbb{N}_0$ will denote the sets of real numbers, non-negative real numbers, positive integers and non-negative integers respectively. We will assume henceforth that all random elements discussed throughout this paper are defined on a suitable probability space $(\Omega,\Sigma,\mathbb{P})$. Throughout the paper, the phrases ``almost surely", ``almost every", etc. will mean ``$\mathbb{P}$-almost surely", ``$\mathbb{P}$-almost every", etc., unless otherwise mentioned.

\subsection{Markov Process Notions}\label{subs:markov}

Let $(\mathbb{X},\mathcal{X})$ be a measurable space. A \emph{transition kernel}\index{transition kernel} on $\mathbb{X}$ is a map $\mu:\mathbb{X}\times\mathcal{X}\to\mathbb{R}$\label{nom:real} such that $\mu(\cdot,\mathcal{A})$ is measurable for each $\mathcal{A}\in\mathcal{X}$ and $\mu(x,\cdot)$ is a measure on $\mathbb{X}$ for each $x\in\mathbb{X}$. A transition kernel $\mu$ is called a \emph{probability kernel}\index{probability kernel} if $\mu(x,\cdot)$ is a probability measure on $\mathbb{X}$ for each $x\in\mathbb{X}$. A \emph{Markovian family of transition distributions}\index{Markovian family of transition distributions} on $\mathbb{X}$ is a family of probability kernels $(\mu_t)_{t\geq0}$ with $\mu_0(x,\cdot)\equiv\delta_x$\label{nom:equiv} for each $x\in\mathbb{X}$ and with $\int_\mathbb{X}\mu_t(y,\mathcal{A})\mu_s(x,dy)=\mu_{s+t}(x,\mathcal{A})$ for all $\mathcal{A}\in\mathcal{X}$, $x\in\mathbb{X}$, and $s,t\geq0$. The associated \emph{transition semigroup of operators}\index{transition semigroup} $(P_t)_{t\geq0}$ of $(\mu_t)_{t\geq0}$ is given by,\label{nom:def} $P_tf(x):=\int_\mathbb{X}f(y)\mu_t(x,dy)$, defined for all bounded measurable $f:\mathbb{X}\to\mathbb{R}$. The \emph{generator}\index{generator} of this  semigroup is given by,
\begin{equation}
Gf:=\lim\limits_{t\to0}\frac{P_tf-f}{t},
\end{equation}
defined for all bounded measurable $f:\mathbb{X}\to\mathbb{R}$ for which the uniform limit exists.
 
Now, let $(\mathbb{X},\mathcal{X})$ be a topological space equipped with the Borel $\sigma$-algebra, $(\mu_t)_{t\geq0}$ be a Markovian family of transition distributions on $\mathbb{X}$, and $(\mathcal{G}_t)_{t\geq0}$ be a filtration. A \emph{Markov family of processes}\index{Markov family of processes} on $\mathbb{X}$ with transition distributions $(\mu_t)_{t\geq0}$ is a family of stochastic processes $\big((\xi_t^x)_{t\geq0}\big)_{x\in \mathbb{X}}$, with c\`adl\`ag (i.e., right-continuous with left limits) sample paths, adapted to $(\mathcal{G}_t)_{t\geq0}$, and satisfying $\mathbb{P}\big(\xi_0^x=x\big)=1$ and $\mathbb{P}\big(\xi_{s+t}^x\in \mathcal{A}|\mathcal{G}_s\big)=\mu_t(\xi_s^x,\mathcal{A})$ for each $(x,\mathcal{A})\in \mathbb{X}\times\mathcal{X}$ and each $t,s\geq0$.

\subsection{Graph Theoretic Notions}\label{subs:graph}
Let $\mathcal{V}=(V,E)$ be an undirected, locally finite, countable, simple graph with no loops. The vertices of $\mathcal{V}$ will often be referred to as \emph{sites}. For any $v\in V$, we denote the degree of $v$ in the graph structure of $\mathcal{V}$ as $\deg_\mathcal{V}(v)$. For any $u,v\in V$, we will write $u\sim v$\label{nom:sim} and $u\not\sim v$\label{nom:nsim} to respectively denote $(u,v)\in E$ or $(u,v)\not\in E$. For any $W\subset V$, we will write $W<V$\label{nom:finsub} to denote that $W$ is finite. For any $v,w\in V$, we define $\text{dist}_{\mathcal{V}}(v,w)$\label{nom:dist} as the graph distance between $v$ and $w$ in the graph structure of $\mathcal{V}$. {For $v\in V$, by $\mathcal{N}_{\mathcal{V},v}$\label{nom:nbdv} and $\mathcal{N}_{\mathcal{V},v}^+$\label{nom:nbdv+} we respectively denote the \emph{neighbourhood} and the \emph{$2$-neighbourhood} of $v$, i.e., $\mathcal{N}_{\mathcal{V},v}:=\{v\}\cup\big\{w\big|w\in V\text{ and }v\sim w\big\}$, and $\mathcal{N}_{\mathcal{V},v}^+:=\big\{w\big|w\in V\text{ and }\mathcal{N}_v\cap\mathcal{N}_w\neq\emptyset\big\}$. Similarly for any $W\subset V$, we define the \emph{neighbourhood} and \emph{$2$-neighbourhood} of $W$ respectively, as $\mathcal{N}_{\mathcal{V},W}:=\bigcup_{v\in W}\mathcal{N}_{\mathcal{V},v}$\label{nom:nbdW} and $\mathcal{N}_{\mathcal{V},W}^+:=\bigcup_{v\in W}\mathcal{N}_{\mathcal{V},v}^+$\label{nom:nbdW+}.} We drop the subscript $\mathcal{V}$ from all the above notations when the underlying graph $\mathcal{V}$ is understood.
 
The \emph{square graph}\index{square graph} $\mathcal{V}^2$\label{nom:V+} of $\mathcal{V}$ is a graph with vertex set $V$ and edge set obtained by adding edges between all $2$-neighbours in $\mathcal{V}$, while retaining all the edges already in $E$. If $\mathcal{V}$ is locally finite, then $\mathcal{V}^2$ is locally finite as well. For any $v\in V$ and $n\in\mathbb{N}$\label{nom:natural}, by $\text{SAW}_{\mathcal{V},n}(v)$\label{nom:sawvn} we denote the set of all $n$-length self-avoiding walks (SAWs) starting from $v$ in $\mathcal{V}$ and $\text{SAW}_\mathcal{V}(v):=\cup_{n\in\mathbb{N}}\text{SAW}_{\mathcal{V},n}(v)$\label{nom:sawv}. The \emph{local connective constant}\index{local connective constant} of $\mathcal{V}$ at $v\in V$ is defined as
\label{nom:lcc}$\Delta_{\mathcal{V},v}:=\limsup_{n\to\infty}\big|\text{SAW}_{\mathcal{V},n}(v)\big|^{\frac{1}{n}}$.

\section{Existence and Convergence of Interacting Particle Systems}\label{sec:excon}
In Subsection \ref{subs:gen} we describe our general framework -- the underlying graph and the family of Markov processes. In Subsection \ref{subs:excon} we state and discuss the main result of this paper on existence and convergence of interacting particle systems on deterministic graphs -- Theorem \ref{thm:exist_cvg_ips}.

\subsection{General Framework}\label{subs:gen}
We will adapt the framework introduced in \cite[Section $2.1$]{Penrose2008existence}. \\

\paragraph*{\textbf{Graph}} The underlying graph $\mathcal{V}=(V,E)$ is an undirected, locally finite, countable (possibly finite), simple graph with vertex set $V$, and with no loop. Throughout the rest of the paper, a ``graph" will always refer to an undirected, countable, simple graph with no loop. \\

\paragraph*{\textbf{State space}} $\mathbb{Y}$\label{nom:lss} is a Polish space which we call the \emph{local state space}\index{local state space}, i.e., the space of configurations at a single site. We endow $\mathbb{Y}^V$ with the product topology. Since $\mathbb{Y}$ is Polish, countable product of Polish spaces being Polish, $\mathbb{Y}^V$ is Polish as well. Our interacting particle systems are families of continuous-time Markov processes on the state space $\mathbb{Y}^V$, which we denote as $(\xi^x)_{x\in\mathbb{Y}^V}:=((\xi_t^x)_{t\geq0})_{x\in\mathbb{Y}^V}$. For any $t\geq0$, any $v\in V$, and any $x\in\mathbb{Y}^V$ $\xi_t^x(v)$ and $\xi_t^x$ denote the local state at $v$ and global state of the process respectively, both at time $t$ and both with initial global state $x$. The neighborhood $\mathcal{N}_v$ represents the set of all sites that can be instantaneously affected by changes at $v$ or that affect the jump rate at site $v$. \\

\paragraph*{\textbf{Jump Rate Kernels and Local Generators}}
In this paper, we discuss only those interacting particle systems which evolve as pure jump processes at each vertex, such that whenever any vertex updates its state, it can influence both the jump rates and states of its neighbouring vertices. To model this behaviour, we start with a pure jump process on the neighbourhood of each vertex of $\mathcal{V}$. In other words, for each $v\in V$, we are given a transition kernel $\alpha_v^*:=\alpha_v^*(\cdot,\cdot)$ on $\mathbb{Y}^{\mathcal{N}_v}$, called the \emph{(local) jump rate kernel at}\index{jump rate kernel} $v$. We note at this point that these jump rate kernels are not, in general, probability kernels. For any $v\in V$, $\alpha_v^*$ corresponds to the generator $G_v^*$ of a pure jump process on $\mathbb{Y}^{\mathcal{N}_v}$, given by,
\begin{equation}\label{eq:locgendef}
G_v^*f(x):=\int_{\bBY^{\mathcal{N}_v}}\big(f(y)-f(x)\big)\alpha_v^*(x,dy),
\end{equation}
for every bounded measurable $f: \bBY^{\mathcal{N}_v}\to\mathbb{R}$, and every $x \in  \bBY^{\mathcal{N}_v}$. For each $v\in V$ and each $x\in\mathbb{Y}^{\mathcal{N}_v}$, we denote $\alpha_v^*(x):=\alpha_v^*(x,\mathbb{Y}^{\mathcal{N}_v})$, the total measure of $\alpha_v^*(x,\cdot)$. We assume that for every $v\in V$,\label{nom:indicator}
\begin{equation}\label{eq:jrknbd} 
c_v:=\sup_{x\in\mathbb{Y}^{\mathcal{N}_v}}\alpha_v^*(x)<\infty,
\end{equation}
i.e., the jump rates are \emph{sitewise bounded}. If, in addition $\sup_{v\in V}c_v<\infty$, we call the jump rates \emph{uniformly bounded}. We do not assume uniform boundedness of the jump rates -- we recall that this is one of the two key assumptions that we are trying to relax in this paper. \\

\paragraph*{\textbf{Generator on $\bBY^V$}}
Once we have a pure jump process on the neighbourhood of each vertex of $\mathcal{V}$, we need these jump processes to interact. To make this rigorous, for each $v\in V$, we use $G_v^*$'s to define the generator $G_v$ of a jump process on $\bBY^V$ that is ``frozen outside $\mathcal{N}_v$", in the following manner. Let $f:\bBY^V\to\mathbb{R}$ be a bounded measurable function for which there exists $A<V$, such that for any $x,y\in\mathbb{Y}^V$, if $x(v)=y(v)$ for every $v\in A$, then $f(x)=f(y)$. We call such a function a \emph{cylinder function with support} $A$. We denote the class of all cylinder functions with support $A<V$ as $\mathcal{C}(A)$. We further denote $\mathcal{C}:=\cup_{A<V}\mathcal{C}(A)$, and simply call it the class of \emph{cylinder functions} on $\mathbb{Y}^V$, i.e., the class of bounded, measurable, real-valued functions on $\mathbb{Y}^V$ which depend on finitely many coordinates.\label{nom:fincrd}\label{nom:Acrd} Let for any $x\in\mathbb{Y}^V$ and $A\subset V$, $x|_A$ denote the restriction of $x$ on $A$. Given $v\in V$, $x\in \bBY^V$, and $y\in \bBY^{\mathcal{N}_v}$, let $x|v|y$ be the element of $\bBY^V$ which agrees with $x$ outside $\mathcal{N}_v$, and with $y$ inside $\mathcal{N}_v$. For $f\in\mathcal{C}$, let the function $f_x^v:\bBY^{\mathcal{N}_v}\to\mathbb{R}$ be given by $f_x^v(y):=f(x|v|y)$. We set $G_vf(x)=G_v^*f_x^v(x|_{\mathcal{N}_v})$. Then $G_v$ is the generator of a jump process on $\bBY^V$ with jump rate kernel $\alpha_v$, given by,
\begin{equation}\label{eq:jrkglob}
\alpha_v(x,x|v|dy)=\alpha_v^*(x|_{\mathcal{N}_v}, dy);\text{ and }\alpha_v(x,dz)=0\text{ if }z|_{V\setminus\mathcal{N}_v}\neq x|_{V\setminus\mathcal{N}_v}
\end{equation}
The family $(\alpha_v)_{v\in V}$ will be referred to as a family of \emph{global jump rate kernels} on $\mathcal{V}$. Although these kernels are defined on the global state space $\mathbb{Y}^V$, in effect they are the same as the local ones, with the global jump rate kernel $\alpha_v$ only prescribing jumps in $\mathbb{Y}^{\mathcal{N}_v}$. In fact, we can show that for every $v\in V$,
\begin{equation}\label{eq:jrknbdglob}
c_v=\sup_{x\in\mathbb{Y}^V}\alpha_v(x)<\infty.
\end{equation}
For $v\in V$, $f\in\mathcal{C}$ and $x\in\mathbb{Y}^V$, by \eqref{eq:locgendef},
\begin{align}
&G_vf(x)=G_v^*\big(f_x^v(x|_{\mathcal{N}_v})\big)=\int_{\bBY^{\mathcal{N}_v}}\big(f_x^v(y)-f_x^v(x|_{\mathcal{N}_v})\big)\alpha_v^*(x|_{\mathcal{N}_v},dy)\nonumber\\&=\int_{\bBY^{\mathcal{N}_v}}\big(f(x|v|y)-f(x)\big)\alpha_v^*(x|_{\mathcal{N}_v}, dy).\label{eq:genclosed}
\end{align}
Intuitively, $G_v$ is the generator of a pure jump process on $\mathbb{Y}^V$ that evolves only inside $\mathbb{Y}^{\mathcal{N}_v}$. In the following lemma we recall a few basic properties of the operator $G_v$ for $v\in V$, from \cite[Lemma $6.1$]{Penrose2008existence}, with \eqref{eq:genbd} slightly modified to account for the fact that we do not have the assumption of uniformly bounded jump rates.
\begin{lemma}
For any $A<V$ and any $f\in\mathcal{C}(A)$,
\begin{equation}
G_vf\equiv0,\forall v\in V\setminus\mathcal{N}_A,\label{eq:gen0}
\end{equation}
\begin{equation}
G_vf\in\mathcal{C}(A\cup\mathcal{N}_v),\forall v\in\mathcal{N}_A,\label{eq:gendom}
\end{equation}
\begin{equation}
||G_vf||\leq2c_{\mathcal{N}_A}||f||,\forall v\in V,\label{eq:genbd}
\end{equation}
where $c_{W}:=\max\limits_{v\in W}\{c_v\}$, for any $W<V$, and $||\cdot||$ denotes the uniform norm.
\end{lemma}
\begin{proof}
The proofs of \eqref{eq:gen0} and \eqref{eq:gendom} are as in \cite[Lemma $6.1$]{Penrose2008existence}. Further, since $\alpha_v^*(x|_{\mathcal{N}_v})\leq c_{\mathcal{N}_A}$ for every $v\in\mathcal{N}_A$, and by \eqref{eq:gen0}, $G_vf\equiv0$ for every $v\in V\setminus\mathcal{N}_A$, \eqref{eq:genclosed} gives us \eqref{eq:genbd}.
\end{proof}

Now, for every $W\subset V$, we define the operator $G_W$ on $\mathcal{C}$ by,
\begin{equation} \label{eq:gendef}
G_W:=\sum_{v\in W}G_v.
\end{equation}

We note that $G_W$ is well-defined as a linear operator for every $W\subset V$. In fact, for any $f\in\mathcal{C}$, $f\in\mathcal{C}(A)$ for some $A<V$. Then for every $v\not\in\mathcal{N}_A$, $G_vf\equiv0$ by \eqref{eq:gen0}, and thus $G_Wf=\sum_{v\in\mathcal{N}_A\cap W}G_vf$. We will now look at cases where $W<V$ or $W=V$. We write $G:=G_V$. A natural question that arises at this juncture is whether, for any $W<V$ or for $W=V$, $G_W$ is the generator of a family of Markov processes on $\mathbb{Y}^V$, i.e., whether there exists a Markovian family of transition kernels, such that the associated semigroup $(P_t)_{t\geq0}$ has generator $G$ on $\mathcal{C}$. If $W<V$, $G_W$ is the generator of a family of jump processes with jump kernel $\sum_{v\in W}\alpha_v$, which is frozen outside $\mathcal{N}_W$. We write $(P_t^W)_{t\geq0}$ for the associated semigroup. Proving the existence of interacting particle systems that evolve on a finite graph is thus straightforward. However, whether $G$ is the generator of a family of Markov processes on $\mathbb{Y}^V$ when $V$ is infinite is a non-trivial question to ask. Theorem \ref{thm:exist_cvg_ips} answers this question in the affirmative by guaranteeing the existence of a family of Markov processes corresponding to $G$, even when the graph is infinite. \\

\paragraph*{\textbf{Self-Updating Systems}}\index{self-updating systems}
In case the jump rate kernel at each $v\in V$ satisfies the following extra condition,
\begin{equation}\label{con:self_upd}
\alpha_v^*(x,{\mathcal{A}})=0\text{, if }\forall y\in {\mathcal{A}},\exists w\in V,\text{ such that }w\neq v\text{ and }x(w)\neq y(w),
\end{equation}
we say that the jump rate kernels are \emph{self-updating}\index{jump rate kernel!self-updating}. Condition \eqref{con:self_upd} ensures that any site $v\in V$ updates only itself and not its neighbours, possibly depending on the states of its neighbours. This is a case of interest, as many examples satisfy \eqref{con:self_upd}.

\subsection{Main Result for Deterministic Graphs}\label{subs:excon}
In this subsection we state and discuss one of the main results of this paper, Theorem \ref{thm:exist_cvg_ips}. As already mentioned in Section \ref{sec:intro}, the main objective of this paper is to relax the assumptions of uniformly bounded degrees and jump rates. We do this in Theorem \ref{thm:exist_cvg_ips}, by replacing these two assumptions with the assumption that certain quantities called double jump trail rates are pointwise finite. To give some intuition, the double jump trail rate at any vertex can be viewed as a weighted version of the local connective constant at that vertex in the graph structure of the square graph $\mathcal{V}^2$ on $\mathcal{V}$, involving not all but a special class (see \eqref{eq:saw2remsaw}) of SAWs emanating from that vertex, the weight associated with a SAW being the product of the $c_v$'s (see \eqref{eq:jrknbd}) corresponding to all but the end vertices of that SAW. The finitude of double jump trail rates ensures that vertices with high degrees and / or jump rates are not too close to each other, and thus the process does not ``{blow up}''. While the finitude of the double jump trail rates ensures the existence of any particle system (including self-updating ones), as we will see in Theorem \ref{thm:exist_cvg_ips}, the finitude of a relatively simpler version of the double jump trail rates that involves $\mathcal{V}$ and not $\mathcal{V}^2$, can guarantee the existence of self-updating systems. Therefore, Theorem \ref{thm:exist_cvg_ips} and its counterpart for random graphs, Theorem \ref{thm:exist_cvg_ips_r}, are versionized both for general systems and self-updating systems. The proofs for the self-updating systems are, however, along the lines of those for the general systems, with some minor changes. Hence we provide detailed proofs of the results for the general systems only, and mention only the necessary logical alterations in their self-updating counterparts. The statement of Theorem \ref{thm:randips}, on the other hand, is the same for both the cases, as it gives conditions only on $\mathcal{V}$ and not $\mathcal{V}^2$, irrespective of whether the system is self-updating or not. 

We now introduce the notions of jump trail rates which will be central to our results. Given $n\in\mathbb{N}$ and any $0\leq k\leq n$, and given an $n$-length SAW $\gamma=v_0,v_1,\cdots,v_n$ in $\mathcal{V}^2$ and an $(n+k)$-length path $\gamma'$ in $\mathcal{V}$, we call $\gamma$ to be a \textit{remnant} of $\gamma'$ if there exist $k$ integers $0\leq i_1<i_2<\cdots<i_k<n$ and $k$ corresponding vertices $u_1,u_2,\cdots,u_k\in V$, such that $v_{i_j}\sim u_j\sim v_{i_j+1}$ for every $1\leq j\leq k$, and the path $\gamma'$ is obtained by plugging these $k$ vertices in between vertices of $\gamma$ as described above.\index{remnant} At this juncture, we note that any SAW $\gamma$ in $\mathcal{V}^2$ is trivially a remnant of a path $\gamma'$ in $\mathcal{V}$. However, $\gamma'$ may not be a SAW in $\mathcal{V}$ itself. This naturally prompts us to define, for any $v\in V$,
\begin{equation}\label{eq:saw2remsaw}
\text{SAW}_{\mathcal{V},n}^*(v):=\Big \{ \gamma\in\text{SAW}_{\mathcal{V}^2,n}(v) \, \Big| \, \exists \gamma'\in\text{SAW}_{\mathcal{V}}(v)\text{ such that }\gamma\text{ is a remnant of }\gamma'\Big\}.
\end{equation}
We further define $\text{SAW}_\mathcal{V}^*(v):=\cup_{n\in\mathbb{N}}\text{SAW}_{\mathcal{V},n}^*(v)$.\label{nom:sawnv*}
\begin{definition}
\label{def:djrt}
\emph{(Jump Trail Rates)} Let $\mathcal{V}=(V,E)$ be a locally finite graph, $\mathbb{Y}$ be a Polish space, and $(\alpha_v^*)_{v\in V}$ be a family of jump rate kernels as in Subsection \ref{subs:gen}, with $c_v$'s as in \eqref{eq:jrknbd}. For every $v\in V$ and $n>1$, we define the \emph{$n$-step double jump trail rate}\index{double jump trail rate!$n$-step} and the \emph{$n$-step simple jump trail rate}\index{simple jump trail rate!$n$-step} from $v$, respectively, as,
\begin{equation}\label{eq:ndjrt}
\Theta_{v,(n)}^*:=\Bigg(\sum\limits_{\substack{\gamma:=v,v_1,\cdots,v_n \\ {\gamma\in\text{SAW}_{\mathcal{V},n}^*(v)}}}\prod\limits_{i=1}^{n-1}c_{v_i}\Bigg)^\frac{1}{n-1}\qquad\text{and}\qquad\Theta_{v,(n)}:=\Bigg(\sum\limits_{\substack{\gamma:=v,v_1,\cdots,v_n \\ {\gamma\in\text{SAW}_{\mathcal{V},n}(v)}}}\prod\limits_{i=1}^{n-1}c_{v_i}\Bigg)^\frac{1}{n-1},
\end{equation}
and the \emph{double jump trail rate}\index{double jump trail rate} and \emph{simple jump trail rate}\index{simple jump trail rate} from $v$, respectively, as,
\begin{equation}\label{eq:djrt}
\Theta_{v}^*:=\limsup\limits_{n\to\infty}\Theta_{v,(n)}^*\qquad\text{and}\qquad\Theta_{v}:=\limsup\limits_{n\to\infty}\Theta_{v,(n)}.
\end{equation} 
\end{definition}

We state our main result on existence and convergence of interacting particle systems on deterministic graphs. We recall, for any $W<V$, $(P_t^W)_{t\geq0}$, the transition semigroup with generator $G_W$.
\begin{theorem}\label{thm:exist_cvg_ips} \emph{(Existence and Convergence for a Deterministic Graph)}
Let $\mathcal{V}=(V,E)$ be a locally finite graph and $\mathbb{Y}$ be a Polish space. Let $(\alpha_v^*)_{v\in V}$ be a family of jump rate kernels such that one of the following assumptions holds:
\begin{enumerate}
    \item[(i)]  For every $v\in V$, $\Theta_{v}^*<\infty$.
    \item[(ii)] The jump rate kernels are self-updating as in \eqref{con:self_upd}, and for every $v\in V$, $\Theta_{v}<\infty$.
\end{enumerate}
{Let} $G:=G_V$ be the operator defined at \eqref{eq:gendef}. 
Then there exists a Markovian family of transition distributions $(\mu_t)_{t\geq0}$ on $\mathbb{Y}^V$, such that the associated semigroup $(P_t)_{t\geq0}$ has generator $G$ on $\mathcal{C}$, i.e.,
\begin{equation}
Gf=\lim\limits_{t\to0}\frac{P_tf-f}{t},
\end{equation}
for every $f\in\mathcal{C}$. Moreover, for any sequence $(W_m)_{m\geq1}$ of finite subsets of $V$, and any $t\geq0$, {if $\liminf\limits_{m\to\infty}(W_m)=V$, then,
\begin{equation}
P_tf(x)=\lim\limits_{m\to\infty}P^{W_m}_tf(x),\;\forall f\in\mathcal{C},x\in \mathbb{Y}^V.\label{eq:convsg}
\end{equation}}
There also exists a filtration $(\mathcal{G}_t)_{t\geq0}$ and a Markov family of processes $(\xi^x)_{x\in \mathbb{Y}^V}$ adapted to $(\mathcal{G}_t)_{t\geq0}$, with transition semigroup $(P_t)_{t\geq0}$.
\end{theorem}
Intuitively, Theorem \ref{thm:exist_cvg_ips} asserts the following in broad strokes -- the geometry of a graph and the assignment of local jump rate kernels at its vertices are enough to guarantee the existence of a family of interacting particle systems on the graph, the generator of which can be obtained by ``compositing" the jump rates appropriately. It is to note that we do not separately mention that the jump rates have to satisfy \eqref{eq:jrknbd} as this is implied by either of the assumptions $(i)$ and $(ii)$ of Theorem \ref{thm:exist_cvg_ips}. We do not give a separate explicit proof of Theorem \ref{thm:exist_cvg_ips} as it follows directly from Propositions \ref{prop:exstcggc} and \ref{prop:ggcexstc} discussed in Section \ref{sec:graph}.

\begin{remark}\label{rem:exist}
Theorem \ref{thm:exist_cvg_ips} complements \cite[Theorem $2.1$]{Penrose2008existence}, as discussed below.
\begin{enumerate}
    \item \textit{(Existence and Convergence)} Theorem \ref{thm:exist_cvg_ips} relaxes the assumptions of uniform degree bound and uniformly bounded jump rates of \cite[Theorem $2.1$]{Penrose2008existence} -- one can verify that these assumptions imply the assumptions of Theorem \ref{thm:exist_cvg_ips}. However, as we will see in the examples discussed in Section \ref{sec:ipsex}, Theorem \ref{thm:exist_cvg_ips} applies to many models which lack a uniform bound on degrees or jump rates. A class of examples can be constructed as follows. We take any graph $\mathcal{V}=(V,E)$ with all local connective constants (defined in Subsection \ref{subs:graph}) finite but no uniform degree bound. An example would be the graph whose vertex set is $T:=\big\{(i,j)\in\mathbb{Z}^2\big|j\leq i,i\in\mathbb{Z}^+\big\}$, and the edges are $(i,0)\sim(i,j)$ and $(i,0)\sim(i\pm1,0)$ for every $i\in\mathbb{N}$ and every $j\leq i$. Theorem \ref{thm:exist_cvg_ips} implies that such a graph will allow self-updating type interacting particle systems with uniformly bounded jump rates. Hence, Theorem \ref{thm:exist_cvg_ips} is an extension of the existence and convergence implication of \cite[Theorem $2.1$]{Penrose2008existence}. Examples of greater interest can be found in random graphs, and will be discussed in detail in Section \ref{sec:rgex}.
    \item \textit{(Uniqueness)} Theorem \ref{thm:exist_cvg_ips} does not claim uniqueness of the system unlike \cite[Theorem $2.1$]{Penrose2008existence}. In \cite[Theorem $2.1$]{Penrose2008existence}, uniqueness is shown under the assumptions of uniformly bounded degrees and jump rates.   
\end{enumerate}
\end{remark}

\section{Interacting Particle Systems on Random Graphs}\label{sec:exconr}

This section discusses the application of Theorem \ref{thm:exist_cvg_ips} to many random graphs. 
Subsection \ref{subs:randfr} is devoted to developing the framework necessary for our agenda and also discussing our main results for the random graphs -- Theorem \ref{thm:exist_cvg_ips_r} and Theorem \ref{thm:randips}. Theorem \ref{thm:exist_cvg_ips_r} is a restatement of Theorem \ref{thm:exist_cvg_ips} for random graphs, and hence we do not give a separate proof of Theorem \ref{thm:exist_cvg_ips_r}. Theorem \ref{thm:randips} gives two readily verifiable sufficient conditions under which the assumptions of Theorem \ref{thm:exist_cvg_ips_r} hold. The first condition concerns the geometry of the random graph -- for a fixed $p>1$ and for each $v\in V$, the $\ell^{1/p}$-sum of the probabilities of self-avoiding walks of length $n$ starting from $v$ grows at most exponentially in $n$. The second condition concerns the assignment of jump rates -- the $p'$-th product moment of the jump rates at any $n$ distinct vertices grows at most exponentially in $n$, where $p'$ is the H\"older conjugate of $p$. We prove Theorem \ref{thm:randips} in Subsection \ref{subs:proofrandips}. Explicit examples of random graphs and interacting particle systems on them satisfying the conditions of Theorem \ref{thm:randips} will be given in Sections \ref{sec:rgex} and \ref{sec:ipsex} respectively.

\subsection{Framework and Main Results for Random Graphs}\label{subs:randfr}
We will now develop the necessary formalism to extend the notions of graph, jump rate kernels, generators, etc. to our desired random setting. We recall the probability space $(\Omega,\Sigma,\mathbb{P})$ from Subsection \ref{subs:gen}. 

\smallskip
\paragraph*{\textbf{Random Graph}}
As has been our custom, throughout the rest of this paper, by ``random graph" we will exclusively refer to undirected, simple random graphs with no loops and with a deterministic vertex set. Let $V$ be a countably infinite set. We can identify the set of all such  undirected, simple graphs without loops with vertex set $V$ with the countable product $\{0,1\}^{V\choose2}$\label{nom:vchoose2}, where ${V\choose2}:=\big\{\{u,v\}\big|\;u,v\in V,\;u\neq v\big\}$. We take the product $\sigma$-algebra $\mathcal{F}$ on $\{0,1\}^{V\choose2}$. Then, by a \emph{random graph with vertex set $V$}, we mean a $\Sigma,\mathcal{F}$-measurable map $\mathcal{V}:\Omega\to\{0,1\}^{V\choose2}$.

\smallskip
\paragraph*{\textbf{Jump Rate Kernels on Random Graph}}
We will now extend the notion of jump rate kernels for random graphs. These will be analogues of the global, and not the local jump rate kernels, a convention purely in order to avoid measurability issues that will be discussed shortly. For every realization $\omega\in\Omega$, for each $v \in V$, we define a transition kernel $\alpha_v(\cdot,\cdot;\omega)$ on $\mathbb{Y}^V$ with the following properties.
\begin{enumerate}
    \item[(i)] For every $x\in\mathbb{Y}^V$ and $\mathcal{A}\in\mathcal{B}(\mathbb{Y}^V)$, $\alpha_v(x,\mathcal{A};\omega)=0$ whenever for each $y\in\mathcal{A}$, there exists $w\not\in\mathcal{N}_{\mathcal V(\omega),v}$, such that $y(w)\neq x(w)$.
    \item[(ii)] For every cylinder function $f : \mathbb{Y}^V \to \mathbb{R}$ and every $x \in \mathbb{Y}^V$, the map, 
    \begin{equation}\label{eq:randjrk}
    \omega \longmapsto\int_{\mathbb{Y}^V}f(y)\,\alpha_v(x,dy;\omega) \, \textrm{is $\Sigma$-measurable}.
    \end{equation}
    \item[(iii)] For every $\omega\in\Omega$, the \emph{total jump rate} map, $x\longmapsto\alpha_v(x;\omega):=\alpha_v(x,\mathbb{Y}^V;\omega)$ is lower semi-continuous (LSC) in the product topology, i.e., $\liminf_{x\to x_0}\alpha_v(x;\omega)\geq\alpha_v(x_o;\omega)$.
\end{enumerate}
Thus, for each $\omega\in\Omega$, $(\alpha_v(\cdot,\cdot;\omega))_{v\in V}$ is a family of global jump rate kernels on the realized graph $\mathcal{V}(\omega)$ as in \eqref{eq:jrkglob}. We then call the correspondence $\omega\mapsto(\alpha_v(\cdot,\cdot;\omega))_{v\in V}$ a family of jump rate kernels on the random graph $\mathcal{V}$, and denote it as $(\boldsymbol{\alpha}_v)_{v\in V}$.
Once we have defined the random jump rate kernels and have their corresponding total jump rates, the most natural task that follows is to define random versions of the $c_v$'s as in \eqref{eq:jrknbd}. We use \eqref{eq:jrknbdglob} to define $c_v(\omega):= \sup_{x \in \mathbb{Y}^{V}}\alpha_v(x;\omega)$ for each $\omega\in\Omega$ and denote the correspondence $\omega\mapsto c_v(\omega)$ as $\boldsymbol{c}_v$ for each $v\in V$. However, since the supremum is being taken on an uncountable collection of random variables  (measurability of $\alpha_v(x;\cdot)$ follows from \eqref{eq:randjrk} with $f\equiv 1$), we need to justify the measurability of $\boldsymbol{c}_v$, which we do in the following result.

\begin{lemma}
For each $v\in V$, $\boldsymbol{c}_v:=\sup_{x \in \mathbb{Y}^{V}}\boldsymbol{\alpha}_v(x)$ is a well-defined $[0,\infty]$-valued random variable.
\end{lemma}
\begin{proof}
We observe that $\mathbb{Y}^V$ being Polish, admits a countable dense subset $D$. Since for each $v\in V$ and each $\omega\in\Omega$, $\alpha_v(\cdot;\omega)$ is LSC,  therefore, using standard topological arguments, $c_v(\omega)=\sup_{x \in D}\alpha_v(x;\omega)$, and thus $\boldsymbol{c}_v$ can be written as a supremum over a countable collection of measurable maps, which guarantees that $\boldsymbol{c}_v$ is a well-defined $[0,\infty]$-valued random variable
\end{proof}
We note at this point that had we defined the random versions of our jump rate kernels to be local ones and not global ones, for each $\omega$ the above argument would have prompted us to obtain a countable dense subset $D(\omega)$ of $\mathcal{N}_{\mathcal{V}(\omega),v}$, and taking a supremum over a set $D(\omega)$ that depends on $\omega$, in general destroys measurability. This is the rationale for defining the random versions of the jump rate kernels as analogues of the global jump rate kernels given in \eqref{eq:jrkglob}, and not the local ones. Once we have defined the $\boldsymbol{c}_v$'s, it is routine to define the random versions of $\Theta_v$ and $\Theta^*_v$, the jump trail rates (to recall \eqref{eq:djrt}). We maintain the convention of using boldface for denoting these random versions, i.e., $\boldsymbol{\Theta}_v$ and $\boldsymbol{\Theta}_v^*$ respectively.

\smallskip
\paragraph*{\textbf{Generator}}
For each realization $\omega \in \Omega$, we define linear operators $G_v(\omega)$ on the algebra $\mathcal C$ of cylinder functions on $\mathbb{Y}^V$, exactly as in Subsection \ref{subs:gen}, using the jump rate kernels $\alpha_v(\omega)$ and the neighbourhood structure induced by $\mathcal V(\omega)$. Explicitly, for $g \in \mathcal C$ and $x \in \mathbb{Y}^V$,
\begin{equation*}
G_v(\omega)g(x):= \int_{\mathbb{Y}^V}\big(g(y) - g(x)\big)\,\alpha_v(x,dy;\omega).
\end{equation*}
Clearly taking $f(\cdot)=g(\cdot)-g(x)$ in \eqref{eq:randjrk} guarantees that $\omega\mapsto G_v(w)g(x)$ is $\Sigma$-measurable for each $x\in\mathbb{Y}^V$, each $v\in V$ and each $g\in\mathcal{C}$. We then define,
\begin{equation}\label{eq:gendef_r}
G(\omega) := \sum_{v \in V} G_v(\omega)   
\end{equation}
which is well-defined on $\mathcal C$, since for each $f \in \mathcal C$ only finitely many terms in the sum are non-zero by \eqref{eq:gen0}. For each fixed $f \in \mathcal C$ and $x \in \mathbb{Y}^V$, the map $\omega \longmapsto G(\omega)f(x)$ is thus $\Sigma$-measurable.
\smallskip

Now that we have introduced the necessary formalism, we will state the following theorem which is an analogue of Theorem \ref{thm:exist_cvg_ips} in random settings.

\begin{theorem}\label{thm:exist_cvg_ips_r} \emph{(General Existence Result for a Random Graph)}
Let $V$ be a countably infinite set and $\mathbb{Y}$ be a Polish space. Let $\mathcal{V}$ be a random graph which is almost surely locally finite and has vertex set $V$ and $(\boldsymbol{\alpha}_v)_{v\in V}$ be a family of jump rate kernels on it such that one of the following assumptions holds:
\begin{itemize}
    \item[(i)] For every $v\in V$, $\mathbb{P}\big(\boldsymbol{\Theta}_{v}^*<\infty\big)=1$.
    \item[(ii)] The jump rate kernels are self-updating as in \eqref{con:self_upd} almost surely, and for every $v\in V$, $\mathbb{P}\big(\boldsymbol{\Theta}_{v}<\infty\big)=1$.
\end{itemize}
For any $\omega\in\Omega$, let $G(\omega)$ be as defined in \eqref{eq:gendef_r}. Then for almost every $\omega$, there exists a Markovian family of transition distributions $(\mu_t^{(\omega)})_{t\ge0}$ on $\mathbb{Y}^V$, with associated transition semigroup $(P_t^{(\omega)})_{t\ge0}$, such that $G(\omega)$ is the generator of $(P_t^{(\omega)})_{t\ge0}$ on $\mathcal C$, i.e.,
\begin{equation}
G(\omega)f=\lim\limits_{t\to0}\frac{P_t^{(\omega)}f-f}{t},
\end{equation}
for every $f\in\mathcal{C}$. Moreover, for any sequence $(W_m)_{m\geq1}$ of finite subsets of $V$, and any $t\geq0$, {if $\liminf\limits_{m\to\infty}(W_m)=V$, then,
\begin{equation}
P_t^{(\omega)}f(x)=\lim\limits_{m\to\infty}P^{W_m,(\omega)}_tf(x),\;\forall f\in\mathcal{C},x\in \mathbb{Y}^V.
\end{equation}}
For every realization where the above statement holds, there also exists a filtration $(\mathcal{G}_t^{(\omega)})_{t\geq0}$ and a Markov family of processes $(\xi^{x,(\omega)})_{x\in \mathbb{Y}^V}$ adapted to $(\mathcal{G}_t^{(\omega)})_{t\geq0}$, with transition semigroup $(P_t^{(\omega)})_{t\geq0}$.
\end{theorem}

We do not give a separate proof of Theorem \ref{thm:exist_cvg_ips_r} as it is a restatement of Theorem \ref{thm:exist_cvg_ips} under random settings. The assumptions of Theorem \ref{thm:exist_cvg_ips} hold almost surely in the premise of Theorem \ref{thm:exist_cvg_ips_r}, and for each realization where these assumptions are satisfied the conclusions of Theorem \ref{thm:exist_cvg_ips} hold. Checking the assumptions of Theorem \ref{thm:exist_cvg_ips_r} for examples is, however, non-trivial. This is because the assumptions involve SAWs in the square graph $\mathcal{V}^2$ of $\mathcal{V}$. In Theorem \ref{thm:randips}, we will give some sufficient conditions on $\mathcal{V}$ and not $\mathcal{V}^2$, which are more tractable, directly applicable to examples, and under which the assumptions of Theorem \ref{thm:exist_cvg_ips_r} hold. In Subsection \ref{sec:rgex}, we will give explicit examples of random graphs to demonstrate the applicability of the sufficient conditions of Theorem \ref{thm:randips}.

\begin{theorem}\label{thm:randips}
\emph{(Some Sufficient Conditions for Existence)} Let $V$ be a countably infinite set. Let $\mathcal{V}$ be a random graph which is almost surely locally finite and has vertex set $V$, $\mathbb{Y}$ be a Polish space, $(\boldsymbol{\alpha}_v)_{v \in V}$ be a family of random jump rate kernels on $\mathcal{V}$, and $p>1$ be such that:

\begin{enumerate}
    \item[(i)] For any $v\in V$, there exist $a_v,b_v>0$, such that for any $n\in\mathbb{N}$,    \begin{equation}\label{eq:aeg}\sum\limits_{v_1,v_2,\cdots,v_n\in V\setminus\{v\}}^{\neq}\mathbb{P}\big(\gamma=v,v_1,v_2,\cdots,v_n\in\text{SAW}_{\mathcal{V},n}(v)\big)^{\frac{1}{p}}\leq a_vb_v^n.
    \end{equation}
    \item[(ii)] There exist $C, D > 0$ such that for any $n \in \mathbb{N}$ and any $n$ distinct vertices $v_1, \dots, v_n\in V$,
    \begin{equation}\label{eq:cvprod}
    \mathbb{E}\left[ \prod_{i=1}^n \boldsymbol{c}_{v_i}^{p'} \right] \le C D^n,\text{ where }p' := p/(p-1).
    \end{equation}
\end{enumerate}
Then all the conclusions of Theorem \ref{thm:exist_cvg_ips_r} hold.
\end{theorem}

We will now outline how we will apply Theorem \ref{thm:randips} in practical settings. We will first consider random graphs that are almost surely locally finite and satisfy Condition $(i)$ of Theorem \ref{thm:randips}, i.e., \eqref{eq:aeg}, for some $p>1$. Any random graph that has a uniform degree bound almost surely, i.e., satisfies $\mathbb{P}(\sup_{v\in V}\deg(v)<\infty)=1$, can be shown to satisfy \eqref{eq:aeg}. Further, for $p':=p/(p-1)$ and for some fixed $C,D>0$, we will require our graphs to satisfy the following product moment inequality,
\begin{equation}\label{eq:pm}
\mathbb{E}\Big[\prod_{i=1}^n\deg(v_i)^{p'}\Big]<CD^n,
\end{equation}
for every choice of $n\in\mathbb{N}$, and every corresponding choice of distinct vertices $v_1,v_2,\cdots,v_n$. In Section \ref{sec:rgex}, we give examples of two random graphs which are almost surely locally finite, satisfy \eqref{eq:aeg} and \eqref{eq:pm}, yet may not have a uniform degree bound almost surely. Next, in Section \ref{sec:ipsex}, we will consider models with jump rate kernels defined on almost surely locally finite random graphs satisfying \eqref{eq:aeg} and \eqref{eq:pm}, such that for each vertex $v$ of the graph, $\boldsymbol{c}_v\leq K\deg(v)$ for some $K>0$. This and \eqref{eq:pm} together will imply Condition $(ii)$ of Theorem \ref{thm:randips}, i.e., \eqref{eq:cvprod}. With both the conditions of Theorem \ref{thm:randips} satisfied, this will guarantee for each model the existence of a family of interacting particle systems on the given random graph conforming to the given jump rate kernels, for almost every realization of the model.

\subsection{Proof of Theorem \ref{thm:randips}}\label{subs:proofrandips}
The proof of Theorem \ref{thm:randips} involves showing that the double (resp., simple) jump trail rates (see \eqref{eq:djrt}) are almost surely finite, and thereby applying Theorem \ref{thm:exist_cvg_ips_r}. We do this by first showing that for any $n\in\mathbb{N}$, the expected $n$-step double (resp., simple) jump trail rates (random versions of the ones defined at \eqref{eq:ndjrt}) from any vertex grow at most exponentially in $n$, and then using Borel-Cantelli lemma.

\begin{proof}[Proof of Theorem \ref{thm:randips}.]
We first prove Theorem \ref{thm:randips} for the general systems, i.e., we do not assume the jump rate kernels to be self-updating almost surely. The central step in proving Theorem \ref{thm:randips} is to show that for any $n\geq2$ and any $v\in V$, there exist $P_v,Q_v>0$ such that,
\begin{equation}\label{eq:Xiexp}
\mathbb{E}\Big[\big(\boldsymbol{\Theta}_{v,(n)}^*\big)^{n-1}\Big]\leq P_vQ_v^{n-1},
\end{equation}
where $\boldsymbol{\Theta}_{v,(n)}^*$ is the (random version of the) $n$-step double jump trail rate as defined in \eqref{eq:ndjrt}.
 
Let us temporarily assume that \eqref{eq:Xiexp} holds and complete the proof. To prove Theorem \ref{thm:exist_cvg_ips_r} for general systems, it suffices to show that for each $v\in V$, $\boldsymbol{\Theta}_{v}^*$ as in \eqref{eq:djrt} is almost surely finite under the assumptions of Theorem \ref{thm:randips}. In fact, we will show that for each $v\in V$,
\begin{equation}\label{eq:Xifin}
\mathbb{P}(\boldsymbol{\Theta}_{v}^*\leq Q_v)=1.
\end{equation}
Towards this, we fix $v \in V$. Then, for any $Q'>Q_v$, using Markov's inequality and \eqref{eq:Xiexp},
\begin{equation*}
\mathbb{P}\big(\boldsymbol{\Theta}_{v,(n)}^*\geq Q'\big)=\mathbb{P}\Big(\big(\boldsymbol{\Theta}_{v,(n)}^*\big)^{n-1}\geq(Q')^{n-1}\Big)\leq\frac{\mathbb{E}\big[(\boldsymbol{\Theta}_{v,(n)}^*)^{n-1}\big]}{(Q')^{n-1}}\leq P_v\bigg(\frac{Q_v}{Q'}\bigg)^{n-1}
\end{equation*}
for all $n\geq2$. This implies, $\sum\limits_{n=2}^\infty\mathbb{P}\big((\boldsymbol{\Theta}_{v,(n)}^*)^{n-1}\geq(Q')^{n-1}\big)<\infty$. Then we use Borel-Cantelli lemma on the sequence of events $\big\{(\boldsymbol{\Theta}_{v,(n)}^*)^{n-1}\geq(Q')^{n-1}\big\}_{n\in\mathbb{N}}$ to obtain \eqref{eq:Xifin}, thereby completing the proof of  Theorem \ref{thm:randips} via Theorem \ref{thm:exist_cvg_ips_r}, provided \eqref{eq:Xiexp} holds.

We shall derive \eqref{eq:Xiexp} now. We fix $v\in V$ and $n>1$. We denote $p'=\frac{p}{p-1}$ for $p$ in \eqref{eq:aeg}. Using H\"older's inequality, we derive that,
\begin{align}
& \mathbb{E}\Big[\big(\boldsymbol{\Theta}_{v,(n)}^*\big)^{n-1}\Big] =\mathbb{E}\Bigg[\sum\limits_{\substack{\gamma:=v,v_1,\cdots,v_n \\ \gamma\in\text{SAW}_{\mathcal{V},n}^*(v)}}\prod\limits_{i=1}^{n-1}\boldsymbol{c}_{v_i}\Bigg]\nonumber
\\&=\mathbb{E}\Bigg[\sum\limits_{v_1,v_2,\cdots,v_n\in V\setminus\{v\}}^{\neq}\mathbbm{1}_{\gamma=v,v_1\cdots,v_n\in\text{SAW}_{\mathcal{V},n}^*(v)}\prod\limits_{i=1}^{n-1}{\boldsymbol{c}_{v_i}}\Bigg]\nonumber
\\&\leq\sum\limits_{v_1,v_2,\cdots,v_n\in V\setminus\{v\}}^{\neq}\Bigg[\mathbb{P}\Big(\gamma=v,v_1\cdots,v_n\in\text{SAW}_{\mathcal{V},n}^*(v)\Big) ^{\frac{1}{p}}\cdot\Bigg(\mathbb{E}\bigg[\prod\limits_{i=1}^{n-1}{\boldsymbol{c}_{v_i}^{p'}}\bigg]\Bigg)^{\frac{1}{p'}}\Bigg]\nonumber\\
&\leq{(CD^{n-1})^\frac{1}{p'}}\cdot\sum\limits_{v_1,v_2,\cdots,v_n\in V\setminus \{v\}}^{\neq}\mathbb{P}\Big(\gamma=v,v_1\cdots,v_n\in\text{SAW}_{\mathcal{V},n}^*(v)\Big)^{\frac{1}{p}},\label{eq:someineq}
\end{align}
where the last inequality uses \eqref{eq:cvprod}. Now we compute the $\sum\mathbb{P}\big(\gamma\in\text{SAW}_{\mathcal{V},n}^*(v)\big)^{\frac{1}{p}}$ term in \eqref{eq:someineq} using a path-reduction argument (to reduce the task of looking at SAWs in $\mathcal{V}^2$ to those in $\mathcal{V}$), and then using \eqref{eq:aeg}. We observe that any $\gamma=v,v_1,v_2,\cdots, v_n\in\text{SAW}_{\mathcal{V},n}^*(v)$ is a remnant of an $(m+n)$-length SAW in $\mathcal{V}$ starting from $v$, for some $0\leq m\leq n$. This means that there exist distinct vertices $w_1,w_2,\cdots,w_m\in V\setminus\{v,v_1,v_2,\cdots,v_n\}$ and integers $0\leq k_1<k_2<\cdots<k_m\leq n-1$ such that $v\sim\cdots\sim v_{k_1}\sim w_1\sim v_{k_1+1}\sim\cdots\sim v_{k_2}\sim w_2\sim v_{k_2+1}\cdots\sim v_{k_m}\sim\ w_m\sim v_{k_m+1}\sim\cdots\sim v_n$ (where $v_0:=v$). Thus we have,
\begin{align}\label{eq:union}
&\Big\{\gamma=v,v_1\cdots,v_n\in\text{SAW}_{\mathcal{V},n}^*(v)\Big\}\nonumber\\&=\bigcup_{m=0}^n\;\;\bigcup_{w_1,w_2,\cdots,w_m\in V\setminus\{v,v_1,v_2,\cdots,v_n\}}^{\neq}\;\;\bigcup_{0\leq k_1<k_2<\cdots<k_m\leq n-1}\Big\{v\sim\cdots\sim v_{k_1}\sim w_1\sim v_{k_1+1}\sim\nonumber\\&\qquad\cdots\sim v_{k_2}\sim w_2\sim v_{k_2+1}\cdots\sim v_{k_m}\sim\ w_m\sim v_{k_m+1}\sim\cdots\sim v_n\Big\}.
\end{align}
Using Boole's inequality on \eqref{eq:union}, we get,
\begin{align}\label{eq:bigineq}
&\mathbb{P}\Big(\gamma=v,v_1\cdots,v_n\in\text{SAW}_{\mathcal{V},n}^*(v)\Big)\nonumber\\&
\leq\sum_{m=0}^n\;\;\sum_{w_1,w_2,\cdots,w_m\in V\setminus\{v,v_1,v_2,\cdots,v_n\}}^{\neq}\;\;\sum_{0\leq k_1<k_2<\cdots<k_m\leq n-1}\mathbb{P}\Big(v\sim\cdots\sim v_{k_1}\sim w_1\sim v_{k_1+1}\sim\nonumber\\&\qquad\cdots\sim v_{k_2}\sim w_2\sim v_{k_2+1}\cdots\sim v_{k_m}\sim\ w_m\sim v_{k_m+1}\sim\cdots\sim v_n\Big). 
\end{align}

Using the fact that for $p > 1$ and a non-negative sequence $(\alpha_i)_{i\in\mathbb{N}}$, $\big(\sum_i \alpha_i\big)^{1/p} \leq \sum_i\alpha_i^{1/p}$,
\begin{align}\label{eq:badineq}
& \mathbb{P}\Big(\gamma=v,v_1\cdots,v_n\in\text{SAW}_{\mathcal{V},n}^*(v)\Big)^{\frac{1}{p}} \nonumber \\
& \leq\sum_{m=0}^n\;\;\sum_{w_1,w_2,\cdots,w_m\in V\setminus\{v,v_1,v_2,\cdots,v_n\}}^{\neq}\;\;\sum_{0\leq k_1<k_2<\cdots<k_m\leq n-1} \mathbb{P}\Big(v\sim\cdots\sim v_{k_1}\sim w_1\sim v_{k_1+1}\sim\nonumber\\&\qquad\cdots\sim v_{k_2}\sim w_2\sim v_{k_2+1}\cdots\sim v_{k_m}\sim w_m\sim v_{k_m+1}\sim\cdots\sim v_n\Big)^{\frac{1}{p}}.
\end{align}
Summing \eqref{eq:badineq} on every $n$-tuple with distinct entries in $V\setminus\{v\}$, and interchanging the order of summation in the RHS, we get,
\begin{align}
&\sum\limits_{v_1,v_2,\cdots,v_n\in V\setminus\{v\}}^{\neq} \mathbb{P}\Big(\gamma=v,v_1\cdots,v_n\in\text{SAW}_{\mathcal{V},n}^*(v)\Big)^{\frac{1}{p}}\nonumber\\&
\leq\sum_{m=0}^n\;\;\sum_{0\leq k_1<k_2<\cdots<k_m\leq n-1}\;\;\sum\limits_{v_1,v_2,\cdots,v_n\in V\setminus\{v\}}^{\neq}\;\;\sum_{w_1,w_2,\cdots,w_m\in V\setminus\{v,v_1,v_2,\cdots,v_n\}}^{\neq}\mathbb{P}\Big(v\sim\cdots\sim v_{k_1}\nonumber\\&\nonumber\sim w_1\sim v_{k_1+1}\sim\cdots\sim v_{k_2}\sim w_2\sim v_{k_2+1}\cdots\sim v_{k_m}\sim\ w_m\sim v_{k_m+1}\sim\cdots\sim v_n\Big)^{\frac{1}{p}}\nonumber\\
&=\sum_{m=0}^n\;\sum_{0\leq k_1<k_2<\cdots<k_m\leq n-1}\;\sum\limits_{u_1,u_2,\cdots,u_{m+n}\in V\setminus\{v\}}^{\neq}\mathbb{P}\Big(\gamma=v,u_1,\cdots,u_{m+n}\in\text{SAW}_{\mathcal{V},m+n}(v)\Big)^{\frac{1}{p}}\nonumber\\
&=\sum_{m=0}^{n}{n\choose{m}}\sum\limits_{u_1,u_2,\cdots,u_{m+n}\in V\setminus\{v\}}^{\neq}\mathbb{P}\Big(\gamma=v,u_1,\cdots,u_{m+n}\in\text{SAW}_{\mathcal{V},m+n}(v)\Big)^{\frac{1}{p}},\label{eq:combin}
\end{align}
where the factor ${n}\choose{m}$ is the number of tuples $(k_1,k_2,\cdots,k_m)$ with distinct increasing entries in $\{0,1,2,\cdots,n-1\}$. Using \eqref{eq:aeg} on \eqref{eq:combin}, we get,
\begin{equation}
\sum\limits_{v_1,v_2,\cdots,v_n\in V\setminus\{v\}}^{\neq} \mathbb{P}\Big(\gamma=v,v_1\cdots,v_n\in\text{SAW}_{\mathcal{V},n}^*(v)\Big)^{\frac{1}{p}}\leq\sum_{m=0}^n{n\choose{m}}a_vb_v^{m+n}=a_vb_v^n(1+b_v)^n\label{eq:finineq}.
\end{equation}
Substituting \eqref{eq:finineq} in \eqref{eq:someineq}, we obtain
\begin{align*}
\mathbb{E}\Big[\big(\boldsymbol{\Theta}_{v,(n)}^*\big)^{n-1}\Big]\leq a_vb_v^n(1+b_v)^n{\big(CD^{n-1}\big)^{\frac{1}{p'}}}=P_vQ_v^{n-1},
\end{align*}
where $P_v={C^{\frac{1}{p'}}a_vb_v(1+b_v)}$ and $Q_v={D^{\frac{1}{p'}}b_v(1+b_v)}$. This proves \eqref{eq:Xiexp}.

Now we consider self-updating systems. {Analogous to \eqref{eq:Xiexp} in} the proof for the general case, the crux of the proof {in the self-updating case} is to show that, for any $n>1$ and any $v\in V$, there exist $P_v',Q_v'>0$ such that,
\begin{equation}\label{eq:Xi'exp}
\mathbb{E}\big[\boldsymbol{\Theta}_{v,(n)}^{n-1}\big]\leq P_v'Q_v'^{n-1},
\end{equation}
where $\boldsymbol{\Theta}_{v,(n)}$ is the (random version of) $n$-step simple jump trail rate as in \eqref{eq:ndjrt}. The rest of the proof then follows by the use of Borel-Cantelli lemma, similar to the general case.

Similar to how we arrived at \eqref{eq:someineq}, we have, for $p'=\frac{p}{p-1}$,
\begin{align}\label{eq:oneineq'}
&\mathbb{E}\big[\boldsymbol{\Theta}_{v,(n)}^{n-1}\big]\leq {\big(CD^{n-1}\big)^\frac{1}{p'}}\cdot\sum\limits_{v_1,v_2,\cdots,v_n\in V\setminus\{v\}}^{\neq} \mathbb{P}\Big(\gamma=v,v_1\cdots,v_n\in\text{SAW}_{\mathcal{V},n}(v)\Big)^{\frac{1}{p}}.
\end{align}
Then we directly apply \eqref{eq:aeg} to \eqref{eq:oneineq'} to obtain \eqref{eq:Xi'exp}, with $P_v'={C^\frac{1}{p'}a_vb_v}$ and $Q_v'={D^{\frac{1}{p'}}b_v}$. Note that we do not need a path-reduction argument as in the general case, as \eqref{eq:oneineq'} already involves $\text{SAW}_{\mathcal{V},n}(v)$ and not $\text{SAW}_{\mathcal{V},n}^*(v)$.
\end{proof}


\section{Examples of Random Graphs Satisfying Condition $(i)$ of Theorem \ref{thm:randips}}\label{sec:rgex}
In this section we will discuss two random graph models which satisfy Condition $(i)$ of Theorem \ref{thm:randips}, i.e., \eqref{eq:aeg}, thus, by the application of Theorem \ref{thm:randips}, setting the stage for the construction of interacting particle systems on almost every realization of such random graphs. Both of these models can lack a uniform degree bound with a positive probability. We recall that a ``random graph" in this paper always refers to an undirected, simple random graph with no loops and with a deterministic vertex set. 

\subsection{Long-Range Percolation}\label{subs:lrp}
Let $V$ be a countable set of points with a family of edge probabilities $\boldsymbol{P}=(P_{v,w})_{v,w\in V;v\neq w}$, i.e., for every $v,w\in V$ with $v\neq w$, $v$ and $w$ have an edge with probability $P_{v,w}$. The probabilities are symmetric, i.e., for every $u,v\in V$ with $u\neq v$, $P_{u,v}=P_{v,u}.$ Let, for some $p\geq1$, the family $\boldsymbol{P}$ be \emph{uniformly $p$-summable}, i.e., $\sup_{u \in V}\sum_{v \in V \setminus \{u\}}P_{u,v}^{1/p}<\infty$. 
This defines a random graph $\mathcal{V}$ with $V$ as its vertex set, which we refer to as the \emph{long-range percolation on} $V$ \emph{with edge probabilities} $\boldsymbol{P}$ \emph{and exponent} $1/p$. 
In order for Condition $(i)$ of Theorem \ref{thm:randips}, i.e., \eqref{eq:aeg} to apply we need uniform $p$-summability for some $p>1$, and we prove this now.

\begin{prop} \label{prop:lrpaeg}
For any countable set $V$, a long-range percolation $\mathcal{V}$ on $V$ with edge probabilities $\boldsymbol{P}$ and exponent $1/p$ satisfies Condition $(i)$ of Theorem \ref{thm:randips}, i.e., \eqref{eq:aeg} if $p>1$. Specifically, for any vertex $v$, any $n\in\mathbb{N}$, and any $p>1$ for which $\boldsymbol{P}$ is uniformly $p$-summable,
\begin{equation*}
\sum\limits_{v_1,v_2,\cdots,v_n\in V\setminus\{v\}}^{\neq}\mathbb{P}\big(\gamma=v,v_1,v_2,\cdots,v_n\in\text{SAW}_{\mathcal{V},n}(v)\big)^{\frac{1}{p}}\leq P^n,
\end{equation*}
where $P:=\sup\limits_{u \in V}\sum\limits_{v \in V \setminus \{u\}}P_{u,v}^{\frac{1}{p}}<\infty$.
\end{prop}
Proposition \ref{prop:lrpaeg} renders the long-range percolation model suitable for the construction of interacting particle systems. This is guaranteed by Theorem \ref{thm:randips}.
\begin{proof}[Proof of Proposition \ref{prop:lrpaeg}]
For any $v\in V$ and any $n\in\mathbb{N}$,
\begin{align}
&\sum\limits_{v_1,v_2,\cdots,v_n\in V\setminus\{v\}}^{\neq}\mathbb{P}\big(\gamma=v,v_1,v_2,\cdots,v_n\in\text{SAW}_{\mathcal{V},n}(v)\big)^{\frac{1}{p}}\nonumber\\
&\leq\sum\limits_{v\neq v_1\neq v_2\neq\cdots\neq v_n}\mathbb{P}\big(v\sim v_1\big)^{\frac{1}{p}}\mathbb{P}\big(v_1\sim v_2\big)^{\frac{1}{p}}\cdots\mathbb{P}\big(v_{n-1}\sim v_n\big)^{\frac{1}{p}}\qquad\text{(using edge independence)}\nonumber\\
&=\sum\limits_{v_1\in V\setminus\{v\}}P_{v,v_1}^{\frac{1}{p}}\sum\limits_{v_2\in V\setminus\{v_1\}}P_{v_1,v_2}^{\frac{1}{p}}\cdots\sum\limits_{v_n\in V\setminus\{v_{n-1}\}}P_{v_{n-1},v_n}^{\frac{1}{p}}\leq P^n,\label{eq:uninteg}
\end{align}
where \eqref{eq:uninteg} follows from uniform $p$-summability.
\end{proof}

In practice, $V$ is usually the vertex set of a deterministic graph, hence the name \textit{long-range percolation}\index{long-range percolation}. A special case is when $V$ is the vertex set of a deterministic \emph{quasi-transitive graph}, i.e., there exists a finite \textit{representative set of vertices} $\{v_1,v_2,\cdots,v_n\}$, such that for any $v\in V$, there exists a graph automorphism $\sigma$, such that $\sigma v=v_i$ for some $i\in\{1,2,\cdots,n\}$. This naturally partitions $V$ into $n$ sets of vertices $\langle v_1\rangle\ni v_1,\langle v_2\rangle\ni v_2,\cdots,\langle v_n\rangle\ni v_n$. We can then define edge probabilities on such a quasi-transitive graph satisfying the following two properties which together guarantee uniform $p$-summability:
\begin{itemize}
\item[(i)] \emph{(automorphism-invariance)} $P_{\sigma u,\sigma v}=P_{u,v}$ for all $u\neq v$ and all graph automorphisms $\sigma$.
\item[(ii)] \emph{($p$-summability)} For some $p\geq1$, $\max\limits_{i=1}^n\sum\limits_{v\in V\setminus\{v_i\}}P_{v_i,v}^{\frac{1}{p}}<\infty$, where the vertices $(v_i)_{1\leq i\leq n}$ are as defined above.
\end{itemize}
 
This has a further special case, where $V$ is the vertex set of a deterministic \emph{transitive} graph, i.e., the representative set of vertices is a singleton set, and we denote its single element as $\textbf{0}$. The condition of $p$-summability then further reduces to $\sum\limits_{v\in V\setminus\{\textbf{0}\}}P_{\textbf{0},v}^\frac{1}{p}<\infty$. This special case is discussed in \cite{duminil2015newproof}, {wherein the edge probabilities are given as $P_{v,w}:=1-e^{-\beta J_{v,w}}$ for every $v,w\in V$ with $v\neq w$, where $(J_{v,w})_{v,w\in V; v\neq w}$ is a collection of non-negative constants symmetric in $(v,w)$, referred to in the paper as \emph{coupling constants}}. Next, we show that the long-range percolation on $\mathbb{Z}^d$ with the nearest-neighbour graph structure can lack a uniform degree bound almost surely.

\begin{prop}\label{prop:noubdeg}
Let $\mathcal{V}$ be the long-range percolation on $\mathbb{Z}^d$ with nearest-neighbour graph structure, such that all the edge probabilities are strictly positive. Then $\mathcal{V}$ lacks a uniform degree bound almost surely, i.e., $\mathbb{P}\big(\sup_{v\in\mathbb{Z}^d}\deg_\mathcal{V}(v)=\infty\big)=1$.
\end{prop}

\begin{proof}
Since all the edge probabilities are strictly positive, one can show that for any $k\in\mathbb{N}$, there exists a deterministic $N(k)\in\mathbb{N}$, such that the box $\big[-N(k),N(k)\big)^d$ contains at least $k$ vertices all of which share an edge (in the long-range percolation graph structure) with $\textbf{0}$ with a positive probability, i.e., 
\begin{equation}\label{eq:boxat0}
    p_\mathbf{0}:=\mathbb{P}\Big(\Big|\big[-N(k),N(k)\big)^d\cap\mathcal{N}_{\mathcal{V},\textbf{0}}\Big|\geq k\Big)>0.
\end{equation}
By automorphism-invariance of the edge probabilities (as $\mathbb{Z}^d$ with the nearest-neighbour structure is a transitive graph), it follows that for any $v\in\mathbb{Z}^d$,
\begin{equation}\label{eq:boxshift}
    \mathbb{P}\Big(\Big|\big([-N(k),N(k))^d+v\big)\cap\mathcal{N}_{\mathcal{V},v}\Big|\geq k\Big)=p_\mathbf{0}>0.
\end{equation}
Using this, for each $k\in\mathbb{N}$, we perform a thinning on $\mathcal{V}$ as follows. We first tessellate $\mathbb{R}^d$ with the box $\big[-N(k),N(k)\big)^d$ and its all possible disjoint translates. We denote the set of vertices which lie at the centers of these boxes as $\Xi_k$. Thus, $\textbf{0}\in\Xi_k$. Next, we delete all the edges of $\mathcal{V}$ that cross any boundary between any two of the boxes. If we denote the degrees of the resultant thinned graph as $\deg_k(v)$ for each $v\in\mathbb{Z}^d$, then for each $v\in\mathbb{Z}^d$, $\deg_k(v)\leq\deg_{\mathcal{V}}(v)$. Also, for any $v,w\in\mathbb{Z}^d$, if $v$ and $w$ lie in two disjoint boxes, $\deg_k(v)$ and $\deg_k(w)$ are independent random variables. In particular $\{\deg_k(v)\}_{v\in{\Xi_k}}$ is a collection of independent random variables such that for each $v\in\Xi_k$, $\mathbb{P}\big(\deg_k(v)\geq k\big)=p_\mathbf{0}>0$. Thus, by the second Borel-Cantelli Lemma, $\mathbb{P}\big(\deg_k(v)\geq k\text{ for some }v\in\Xi_k\text{ i.o.}\big)=1$. Since $\deg_{\mathcal{V}}\geq\deg_k$, this implies, for each $k\in\mathbb{N}$, $\mathbb{P}\big(\deg_\mathcal{V}(v)\geq k\text{ for some }v\in\mathbb{Z}^d\text{ i.o.}\big)=1$. This implies $\mathbb{P}\Big(\bigcap\limits_{k\in\mathbb{N}}\big\{\deg_\mathcal{V}(v)\geq k\text{ for some }v\in\mathbb{Z}^d\text{ i.o.}\big\}\Big)=1$, which in turn implies $\mathbb{P}\big(\sup\limits_{v\in \mathbb{Z}^d}\deg_\mathcal{V}(v)=\infty\big)=1$.
\end{proof}

We finally prove the product moment inequality given in \eqref{eq:pm} for the long-range percolation.

\begin{prop}\label{prop:pmlrp}
For a long-range percolation $\mathcal{V}$ on any countable set $V$ with edge probabilities $\boldsymbol{P}$ and exponent $1/p$ for some $p>1$, there exist $C,D>0$, such that for $p':=p/(p-1)$, for any $n\in\mathbb{N}$ and any $n$ distinct vertices $v_1,v_2,\cdots,v_n\in V$,
\begin{equation}\label{eq;pmlrp}
\mathbb{E}\Big[\prod_{i=1}^n\deg(v_i)^{p'}\Big]\leq CD^n.
\end{equation}
\end{prop}
\begin{proof}
Let for any $u,v\in V$, $X_{u,v}:=\mathbbm{1}_{u\sim v}$ denote the indicator random variable for the edge $\{u,v\}$. For any non-negative random variable $Y$ and any $t>0$, we have the inequality $Y^{p'} \leq C_{p'} e^{tY}$ for constant $C_{p'} = (p'/te)^{p'}$, which is derived using basic calculus by finding the maximum value of the function $f(y) = y^{p'} e^{-ty}$. Applying this to each degree term, we obtain,
\begin{equation*}
\prod_{i=1}^n \deg(v_i)^{p'} \leq \prod_{i=1}^n \left( C_{p'} e^{t \cdot \deg(v_i)} \right) = C_{p'}^n \exp\left(t \sum_{i=1}^n \deg(v_i)\right).
\end{equation*}
Taking expectations, we obtain,
\begin{equation}\label{eq:pmlrpexp}
\mathbb{E}\left[\prod_{i=1}^n \deg(v_i)^{p'}\right] \leq C_{p'}^n \mathbb{E}\left[ \exp\left(t \sum_{i=1}^n \deg(v_i)\right) \right].
\end{equation}
The sum of degrees counts every edge incident to the set $\{v_1, \dots, v_n\}$. Let $E_{int}$ be the set of ``internal" edges between vertices in $\{v_1, \dots, v_n\}$ and $E_{ext}$ be the set of ``external" edges connecting a vertex in $\{v_1, \dots, v_n\}$ to a vertex outside. Then we can write,
\begin{equation*}
\sum_{i=1}^n \deg(v_i) = \sum_{i=1}^n \sum_{u \neq v_i} X_{v_i,u} = 2\sum_{\{u,v\} \in E_{int}}  X_{u,v} + \sum_{\{u',v'\} \in E_{ext}}  X_{u',v'},  
\end{equation*}
as the internal edges are counted twice (the internal edge $\{v_i, v_j\}$ contributes to both $\deg(v_i)$ and $\deg(v_j)$) and the external edges are counted once. Next, since edge indicators are mutually independent, the expectation in \eqref{eq:pmlrpexp} factors over edges, and thus we get,
\begin{equation*}
\mathbb{E}\left[ \exp\left(t \sum_{i=1}^n \deg(v_i)\right) \right] = \prod_{\{u,v\} \in E_{int}} \mathbb{E}[e^{2t X_{u,v}}]\;\;\;\;\cdot \prod_{\{u',v'\} \in E_{ext}} \mathbb{E}[e^{t X_{u',v'}}].
\end{equation*}
Using the moment-generating function of a Bernoulli variable and the fact that for any $x>0$, $1+x\leq e^x$ holds, we obtain,
\begin{align*}
&\mathbb{E}\left[ \exp\left(t \sum_{i=1}^n \deg(v_i)\right) \right]\leq \exp\left( \sum_{\{u,v\} \in E_{int}} P_{u,v} (e^{2t} - 1) + \sum_{\{u',v'\} \in E_{ext}} P_{u',v'} (e^t - 1) \right)
\\&\leq \exp\left((e^{2t} - 1)\cdot\Big(\sum_{\{u,v\} \in E_{int}} P_{u,v} + \sum_{\{u',v'\} \in E_{ext}} P_{u',v'}\Big) \right) \leq \exp\Big((e^{2t} - 1)\sum_{i=1}^n \sum_{u \neq v_i} P_{v_i, u} \Big) \leq \big(e^{(e^{2t} - 1)P_0}\big)^n,
\end{align*}
where $P_0:=\sup_{v\in V}\sum_{w\neq v}P_{v,w}<\infty$ by uniform $1$-summability (uniform $p$-summability for $p>1$ implies uniform $1$-summability). Combining with \eqref{eq:pmlrpexp},
\begin{equation*}
\mathbb{E}\left[\prod_{i=1}^n \deg(v_i)^{p'}\right] \leq \left( C_{p'} \cdot e^{(e^{2t}-1)P_0} \right)^n
\end{equation*}
which is \eqref{eq;pmlrp} with $C:=1$ and $D:=C_p \cdot e^{(e^{2t}-1)P_0}$.
\end{proof}

\subsection{Geometric Random Graph}\label{subs:grg}
We consider a metric space $(\mathbb{M},\delta)$. Let $V\subset\mathbb{M}$ be countable, such that every pair of points in it is separated by a minimum distance. We can always set this minimum separating distance to $1$ by appropriate rescaling, i.e., we can assume without loss of generality that for any $v,w\in V$,
\begin{equation}\label{eq:hardcore}
    {\delta(v,w)\geq1}.
\end{equation}
Further, let, for some integer $s>1$,
\begin{equation}\label{eq:moms}
\sup\limits_{v\in V}\sum\limits_{w\in V\setminus\{v\}}\frac{1}{\delta(v,w)^s}<\infty.
\end{equation}
Clearly \eqref{eq:moms} holds for any $s'>s$. Let $p>1$ and 
$Z$, which we call the \textit{grain radius,} be a non-negative random variable such that,
\begin{equation}\label{eq:grain}
\mathbb{E}\big[Z^{s\cdot\max\{p,p'\}}\big]<\infty.
\end{equation}
{Our random graph $\mathcal{V}$, which we call a \textit{geometric random graph with grain radius} $Z$ \emph{and exponent} $s$, is as follows.} At every $v\in V$, we draw a sphere (which we call a \emph{grain} at $v$) with random radius $R_v$, such that $R_v$'s are i.i.d., distributed as $Z$. We put an edge between any two $u,v\in V$ if $u\in\textbf{B}_{R_v}(v)$ and $v\in\textbf{B}_{R_u}(u)$, i.e., if $\delta(u,v)<\min\{R_v,R_u\}$ (for any $R>0$ and any $v\in\mathbb{M}$, $\mathbf{B}_R(v):=\{w:w\in\mathbb{M},\delta(v,w)<R\}$). We now prove that this random graph satisfies Condition $(i)$ of Theorem \ref{thm:randips}, i.e., \eqref{eq:aeg} for any $p$ described above.

\begin{prop}\label{prop:grgaeg}
{Let $V$ be any countable subset of a metric space such that \eqref{eq:hardcore} and \eqref{eq:moms} hold for some integer $s>1$, and $\mathcal{V}$ be a geometric random graph on $V$ with grain radius $Z$ and exponent $s$. Then $\mathcal{V}$ satisfies Condition $(i)$ of Theorem \ref{thm:randips}, i.e., \eqref{eq:aeg} for any $p>1$ for which \eqref{eq:grain} holds. Specifically, for any vertex $v\in V$, any $n\in\mathbb{N}$ and any $p>1$ for which \eqref{eq:grain} holds, }
\begin{equation*}
{\sum\limits_{v_1,v_2,\cdots,v_n\in V\setminus\{v\}}^{\neq}\mathbb{P}\big(\gamma=v,v_1,v_2,\cdots,v_n\in\text{SAW}_{\mathcal{V},n}(v)\big)^{\frac{1}{p}}\leq  \mathbb{E}[Z^{sp}]^{\frac{1}{p}}\big(S\mathbb{E}[Z^{sp}]^{\frac{1}{p}}\big)^n},
\end{equation*}
where $S:=\sup\limits_{v\in V}\sum\limits_{w\in V\setminus\{v\}}\frac{1}{\delta(v,w)^s}<\infty$.
\end{prop}
{Proposition \ref{prop:grgaeg} renders the geometric random graph model suitable for construction of interacting particle systems. This is guaranteed by Theorem \ref{thm:randips}.}

\begin{proof}
Let $s'=sp(>1)$. We first compute the probability of any $n$-length SAW $v,v_1,v_2,\cdots,v_n$ starting from $v$. This probability is,
\begin{align*}
\begin{split}
\mathbb{P}\big(R_v>\delta(v,v_1),R_{v_1}>\max\{\delta(v,v_1),\delta(v_1,v_2)\},R_{v_2}>\max\{\delta(v_1,v_2),\delta(v_2,v_3)\},\cdots,&\\R_{v_{n-1}}>\max\{\delta(v_{n-2},v_{n-1}),\delta(v_{n-1},v_n)\},R_{v_n}>\delta(v_{n-1},v_n)\big),
\end{split}
\end{align*}
which can be rewritten as,
\begin{align}\label{eq:geompr}
\begin{split}
\mathbb{P}\big(R_v>\delta(v,v_1)\big)\mathbb{P}\big(R_{v_1}>\max\{\delta(v,v_1),\delta(v_1,v_2)\}\big)\mathbb{P}\big(R_{v_2}>\max\{\delta(v_1,v_2),\delta(v_2,v_3)\}\big)\cdots&\\\mathbb{P}\big(R_{v_{n-1}}>\max\{\delta(v_{n-2},v_{n-1}),\delta(v_{n-1},v_n)\}\big)\mathbb{P}\big(R_{v_n}>\delta(v_{n-1},v_n)\big),
\end{split}
\end{align}
using the independence of $R_v$'s.
Now, using Markov's inequality,
\begin{align*}
\begin{split}
\eqref{eq:geompr}\leq\frac{\mathbb{E}[R_v^{s'}]}{\delta(v,v_1)^{s'}}\cdot\frac{\mathbb{E}[R_{v_1}^{s'}]}{\max\{\delta(v,v_1),\delta(v_1,v_2)\}^{s'}}\cdot\frac{\mathbb{E}[R_{v_2}^{s'}]}{\max\{\delta(v_1,v_2),\delta(v_2,v_3)\}^{s'}}\cdot\cdots\cdot&\\\frac{\mathbb{E}[R_{v_{n-1}}^{s'}]}{\max\{\delta(v_{n-2},v_{n-1}),\delta(v_{n-1},v_n)\}^{s'}}\cdot\frac{\mathbb{E}[R_{v_n}^{s'}]}{\delta(v_{n-1},v_n)^{s'}}
\end{split}&\\
\begin{split}
\leq\frac{\mathbb{E}[Z^{s'}]}{\delta(v,v_1)^{s'}}\cdot\frac{\mathbb{E}[Z^{s'}]}{\big(\delta(v,v_1)\delta(v_1,v_2)\big)^\frac{s'}{2}}\cdot\frac{\mathbb{E}[Z^{s'}]}{\big(\delta(v_1,v_2)\delta(v_2,v_3)\big)^\frac{s'}{2}}\cdot\cdots\cdot&\\\frac{\mathbb{E}[Z^{s'}]}{\big(\delta(v_{n-2},v_{n-1})\delta(v_{n-1},v_n)\big)^\frac{s'}{2}}\cdot\frac{\mathbb{E}[Z^{s'}]}{\delta(v_{n-1},v_n)^{s'}}
\end{split}&\\
\begin{split}
=\frac{\mathbb{E}[Z^{s'}]^{n+1}}{\delta(v,v_1)^\frac{3s'}{2}\delta(v_1,v_2)^{s'}\delta(v_2,v_3)^{s'}\cdots \delta(v_{n-2},v_{n-1})^{s'}\delta(v_{n-1},v_n)^\frac{3s'}{2}}.
\end{split}
\end{align*}
Therefore, we have, similar to the proof of Proposition \ref{prop:lrpaeg},
\begin{align*}
&\sum\limits_{v_1,v_2,\cdots,v_n\in V\setminus\{v\}}^{\neq}\mathbb{P}\big(\gamma=v,v_1,v_2,\cdots,v_n\in\text{SAW}_{\mathcal{V},n}(v)\big)^{\frac{1}{p}}
\\&\leq\sum\limits_{v\neq v_1\neq\cdots v_n}\Bigg(\frac{\mathbb{E}[Z^{s'}]^{n+1}}{\delta(v,v_1)^\frac{3s'}{2}\delta(v_1,v_2)^{s'}\delta(v_2,v_3)^{s'}\cdots \delta(v_{n-2},v_{n-1})^{s'}\delta(v_{n-1},v_n)^\frac{3s'}{2}}\Bigg)^{\frac{1}{p}}&\\
\begin{split}
=\big(\mathbb{E}[Z^{s'}]\big)^{\frac{n+1}{p}}\sum\limits_{v_1\in V\setminus\{v\}}\frac{1}{\delta(v,v_1)^\frac{3s}{2}}\sum\limits_{v_2\in V\setminus\{v_1\}}\frac{1}{\delta(v_1,v_2)^s}\sum\limits_{v_3\in V\setminus\{v_2\}}\frac{1}{\delta(v_2,v_3)^s}\cdots&\\\sum\limits_{v_{n-1}\in V\setminus\{v_{n-2}\}}\frac{1}{\delta(v_{n-2},v_{n-1})^s}\sum\limits_{v_n\in V\setminus\{v_{n-1}\}}\frac{1}{\delta(v_{n-1},v_n)^\frac{3s}{2}}\qquad\qquad\text{(Since $p=\frac{s'}{s}$)}
\end{split}
\\&\leq \mathbb{E}[Z^{sp}]^{\frac{1}{p}}\big(S\mathbb{E}[Z^{sp}]^{\frac{1}{p}}\big)^n,
\end{align*}
where the last step uses \eqref{eq:moms} and \eqref{eq:grain}.
\end{proof}
 
The most natural example for which \eqref{eq:moms} holds, would be $\mathbb{Z}^d$ (as a subset of the Euclidean metric space $\mathbb{R}^d$, for any $d\in\mathbb{N}$). However we prove \eqref{eq:moms} for a fairly important class of countable subsets of the Euclidean space -- \emph{linearly repetitive Delone sets}. We will prove \eqref{eq:moms} for $\mathbb{Z}^d$ as a part of this proof. A subset $V$ of any metric space $(\mathbb{M},\delta)$ is called a \emph{Delone set} if there exist $r_\text{pack},r_\text{cov} > 0$ with:
\begin{enumerate}
    \item[(i)] \emph{(Relative density)} Each ball of radius $r_\text{cov}$ in $\mathbb{M}$ contains at least one point of $V$.
    \item[(ii)] \emph{(Uniform discreteness)} Each ball of radius $r_\text{pack}$ in $\mathbb{M}$ contains at most one point of $V$.
\end{enumerate}
Uniform discreteness implies that any Delone set has a minimum separating distance $\leq r_{\text{pack}}$, and thus, it can be rescaled so as to satisfy \eqref{eq:hardcore}.  Let $d\geq2$. A Delone set $V$ of $\mathbb{R}^d$ is called \emph{repetitive} if for each $r > 0$, there exists $M > 0$ such that for each $z\in\mathbb{R}^d$, and for each $x\in\mathbb{R}^d$, there exists $y\in V\cap\textbf{B}_M(z)$, such that:
\begin{equation}
V\cap\textbf{B}_r(y)=\big(V\cap\textbf{B}_r(x)\big)+y-x.
\end{equation}
The smallest such $M$ is denoted by $M_V(r)$. If there exists $L > 0$ such that $M_V(r) \leq Lr$, then $V$ is called a \emph{linearly repetitive Delone set}. Some examples of linearly repetitive Delone sets include $\mathbb{Z}^d$, all (translates of) lattices (by a lattice we mean a locally finite subgroup of $(\mathbb{R}^d,+)$; see e.g. \cite{daniele2010lattice}) and various quasicrystal-type aperiodic point sets (see for example \cite{lagarias_2003}). We have the following result, Proposition \ref{thm:delmoms}, which allows us to construct geometric random graphs on linearly repetitive Delone sets. The proof uses the concept of bi-Lipschitz equivalence. Let $(\mathbb{M},\delta)$ be a metric space. Let $V_1,V_2\subset\mathbb{M}$ be countable subsets. Then $V_1$ and $V_2$ are said to be \emph{bi-Lipschitz equivalent} to each other if there exists a homeomorphism $\Phi:V_1\to V_2$ and a constant $K>0$ such that,
\begin{equation}\label{eq:bilip}
\frac{1}{K}\delta(v,v')\leq \delta\big(\Phi(v),\Phi(v')\big)\leq K\delta(v,v')
\end{equation}
for every $v,v'\in V$. Such a homeomorphism is called a \emph{bi-Lipschitz map}.

\begin{prop}\label{thm:delmoms}
For any $d\geq2$, every linearly repetitive Delone set $V$ in $\mathbb{R}^d$ satisfies \eqref{eq:moms} with $s\geq d+1$.
\end{prop}
\begin{proof}
We first prove that \eqref{eq:moms} holds for $\mathbb{Z}^d$. It can be shown that there exists $k>0$, such that {for any $v\in\mathbb{Z}^d$}, $\mathbb{Z}^d\cap\mathbf{B}_k(v)\setminus\{v\}\neq\emptyset$ and the number of points in $\mathbb{Z}^d$ between the distances $r$ and $r+k$ from $v$ is ${\theta(r^{d-1})}$, i.e., $\big|\mathbb{Z}^d\cap\big(\mathbf{B}_{r+k}(v)\setminus\mathbf{B}_r(v)\big)\big|={\theta(r^{d-1})}$.
Then, using local finiteness of $\mathbb{Z}^d$, for any $v\in\mathbb{Z}^d$, we get,
\begin{align*}
S_v&:=\sum\limits_{w\in\mathbb{Z}^d\setminus\{v\}}\frac{1}{|v-w|^{d+1}}\\&\leq\frac{\big|\mathbb{Z}^d\cap\mathbf{B}_k(v)\setminus\{v\}\big|}{\Big(\min\limits_{w\in\mathbb{Z}^d\cap\mathbf{B}_k(v)\setminus\{v\}}|v-w|\Big)^{d+1}}+\sum\limits_{n=1}^\infty\frac{\big|\mathbb{Z}^d\cap\big(\mathbf{B}_{nk+k}(v)\setminus\mathbf{B}_{nk}(v)\big)\big|}{(nk)^{d+1}}<\infty.
\end{align*}
Finally, the fact that $\mathbb{Z}^d$ is invariant under translation by any element of $\mathbb{Z}^d$ itself, guarantees that for any $u,v\in\mathbb{Z}^d$, $S_u=S_v$, and hence, \eqref{eq:moms} holds for $\mathbb{Z}^d$ with $s\geq d+1$. This, along with the facts that for any $d\geq2$, every linearly repetitive Delone set in $\mathbb{R}^d$ is bi-Lipschitz equivalent to $\mathbb{Z}^d$ \cite[Theorem $2.1$]{alisteprieto2013}, and that \eqref{eq:moms} is invariant under any bi-Lipschitz map (Lemma \ref{lem:bilip}), proves Proposition \ref{thm:delmoms}.
\end{proof}

We finally state and prove Lemma \ref{lem:bilip} which is used in the proof of Proposition \ref{thm:delmoms}, and which implies that if two countable subsets of a metric space are bi-Lipschitz equivalent, then if a geometric random graph can be constructed on one of them, then a geometric random graph can be constructed on the other.
\begin{lemma}\label{lem:bilip}
Let $(\mathbb{M},\delta)$ be a metric space. Let $V_1,V_2\subset\mathbb{M}$ be countable subsets that are bi-Lipschitz equivalent to each other. Then, for any $s>1$, \eqref{eq:moms} holds for $V_1$ iff it holds for $V_2$.
\end{lemma}
\begin{proof}
We assume \eqref{eq:moms} holds for $V_1$. Then, using the first inequality of \eqref{eq:bilip} and the fact that $\Phi:V_1\to V_2$ is a homeomorphism,
\begin{align}
&\sup\limits_{v\in V_2}\sum\limits_{w\in V_2\setminus\{v\}}\frac{1}{\delta(v,w)^s}\leq K^s\sup\limits_{v\in V_2}\sum\limits_{w\in V_2\setminus\{v\}}\frac{1}{\delta(\Phi^{-1}(v),\Phi^{-1}(w))^s}\nonumber
\\&=K^s\sup\limits_{v\in V_1}\sum\limits_{w\in V_1\setminus\{\Phi^{-1}(v)\}}\frac{1}{\delta(\Phi^{-1}(v),\Phi^{-1}(w))^s}<\infty,\nonumber
\end{align}
which proves one direction. The converse can be proven using the second inequality of \eqref{eq:bilip}.
\end{proof} 

Next, we show that for any $d\in\mathbb{N}$, the geometric random graph on $\mathbb{Z}^d$ with the Euclidean metric structure, with any choice of $s>1$ and with a grain radius having unbounded support lacks a uniform degree bound almost surely.

\begin{prop}\label{prop:noubdeg2}
Let $\mathcal{V}$ be a geometric random graph on $\mathbb{Z}^d$ with the Euclidean metric structure, with any choice of exponent $s>1$ and with a grain radius having unbounded support. Then $\mathcal{V}$ lacks a uniform degree bound almost surely, i.e., $\mathbb{P}\big(\sup\limits_{v\in\mathbb{Z}^d}\deg_\mathcal{V}(v)=\infty\big)=1$..
\end{prop}

\begin{proof}
Since the grain radius has unbounded support, all the edge probabilities are strictly positive. Now for any $k\in\mathbb{N}$, there exists a deterministic positive integer $L(k)>0$, such that the box $[-L(k),L(k))^d$ contains at least $k$ vertices other than $\mathbf{0}$. We choose $k$ specific vertices $v_1,v_2,\cdots,v_k\in[-L(k),L(k))^d$ and denote the event, $E_k:= \{ R_{\mathbf{0}} > 2L(k)\sqrt{d} , R_{v_1} > 2L(k)\sqrt{d}, \cdots, R_{v_k} > 2L(k)\sqrt{d} \}$. Then independence of the grain radii and the fact that each of them has an unbounded support imply $p_{\mathbf{0}}:=\mathbb{P}(E_k)>0$.
For each $k\in\mathbb{N}$, we perform a thinning on the geometric random graph as we did for the long-range percolation. We first tile $\mathbb{R}^d$ with the box $[-L(k),L(k))^d$ and its all possible disjoint translates. We denote the set of vertices which lie at the centers of these boxes as $\Xi_k$. Thus, $\textbf{0}\in\Xi_k$. Next, we delete all the edges of $\mathcal{V}$ that cross the boundaries between any two of the boxes. If we denote the degrees of the resultant thinned graph as $\deg_k(\cdot)$, then for each $v\in\mathbb{Z}^d$, $\deg_k(v)\leq\deg_{\mathcal{V}}(v)$. Also, if for any $v\in\Xi_k$, we denote the event $    E_{k,v}:= \{ R_v > 2L(k)\sqrt{d} , R_{v_1+v} > 2L(k)\sqrt{d}, \cdots, R_{v_k+v} > 2L(k)\sqrt{d} \}$, then, since the events $(E_{k,v})_{v\in\Xi_k}$ depend on the grain radii of disjoint sets of vertices, the independence of grain radii imply the independence of $(E_{k,v})_{v\in\Xi_k}$. Further $\mathbb{P}(E_{k,v})=p_\mathbf{0}>0$ for each $v\in\Xi_k$, which follows from the fact that the graph is on $\mathbb{Z}^d$, and the grain radii are i.i.d., hence translation-invariant. Thus, by the second Borel-Cantelli Lemma, $\mathbb{P}\big(E_{k,v}\text{ for some }v\in\Xi_k\text{ i.o.}\big)=1$. Further, for any $k\in\mathbb{N}$ and any $v\in\Xi_k$, $E_{k,v}\subseteq\{\deg_k(v)\geq k\}$. Since $\deg_{\mathcal{V}}\geq\deg_k$, we have, for each $k\in\mathbb{N}$,
\begin{equation*}
    \mathbb{P}(E_{k,v})\leq\mathbb{P}(\deg_k(v)\geq k)\leq\mathbb{P}(\deg_\mathcal{V}(v)\geq k),
\end{equation*}
which implies $\mathbb{P}\big(\deg_\mathcal{V}(v)\geq k\text{ for some }v\in\mathbb{Z}^d\text{ i.o.}\big)=1$, and hence $\mathbb{P}\Big(\bigcap\limits_{k\in\mathbb{N}}\big\{\deg_\mathcal{V}(v)\geq k\text{ for some }v\in\mathbb{Z}^d\text{ i.o.}\big\}\Big)=1$, for each $k\in\mathbb{N}$. This in turn implies $\mathbb{P}\big(\sup\limits_{v\in\mathbb{Z}^d}\deg_\mathcal{V}(v)=\infty\big)=1$.
\end{proof}

We finally prove the product moment inequality given in \eqref{eq:pm} for the geometric random graph.

\begin{prop}\label{prop:pmgrg}
Let $V$ be any countable subset of a metric space such that \eqref{eq:hardcore} and \eqref{eq:moms} hold for some $s>1$, and $\mathcal{V}$ be a geometric random graph on $V$ with exponent $s$ and grain radius $Z$ satisfying \eqref{eq:grain} for some $p>1$. Then there exist $C,D>0$, such that for $p':=p/(p-1)$, for any $n\in\mathbb{N}$ and any $n$ distinct vertices $v_1,v_2,\cdots,v_n\in V$,
\begin{equation}\label{eq;pmgrg}
\mathbb{E}\Big[\prod_{i=1}^n\deg(v_i)^{p'}\Big]\leq CD^n.
\end{equation}
\end{prop}
\begin{proof}
For any $v\in V$, $\deg(v)$ is the number of vertices $u\in V\setminus\{v\}$ such that $\delta(u, v) < \min\{R_u, R_v\}$. Since $\min\{R_u, R_v\} \le R_v$, we have,
\begin{equation*}
\deg(v) \leq \overline{R}_v := \sum_{u \neq v} \mathbbm{1}_{\delta(u, v) < R_v}.
\end{equation*}
$\overline{R}_{v_1}, \cdots, \overline{R}_{v_n}$ are mutually independent as they depend only on the i.i.d. radii $R_{v_1}, \dots, R_{v_n}$. Thus,
\begin{equation}\label{eq:pmgrgexp}
\mathbb{E}\left[\prod_{i=1}^n \deg(v_i)^{p'}\right] \leq \prod_{i=1}^n \mathbb{E}[\overline{R}_{v_i}^{p'}]
\end{equation}
Next, we bound $\overline{R}_v$ in terms of $R_v$ using \eqref{eq:moms}. We observe that $\mathbbm{1}_{A < B}\leq (B/A)^s$ whenever $A, B > 0$ and $s > 0$. Applying this to the condition $\delta(u, v) < R_v$, we get,
\begin{equation*}
\overline{R}_v = \sum_{u \neq v}\mathbbm{1}_{\delta(u, v) < R_v} \leq \sum_{u \neq v}\left(\frac{R_v}{\delta(u, v)}\right)^s= R_v^s \sum_{u \neq v} \frac{1}{\delta(u, v)^s}\leq R_v^s\cdot S,
\end{equation*}
where $S:=\sup_{v\in V}\sum_{u\neq v}\frac{1}{\delta(u,v)^s}<\infty$ by \eqref{eq:moms}. Using this in \eqref{eq:pmgrgexp} yields, 
\begin{equation*}
\mathbb{E}\left[\prod_{i=1}^n \deg(v_i)^{p'}\right]\leq\prod_{i=1}^n\mathbb{E}\big[(S R_v^s)^{p'}\big] =\Big(S^{p'}\mathbb{E}\big[Z^{sp'}\big]\Big)^n,
\end{equation*}
where the finiteness of $\mathbb{E}\big[Z^{sp'}\big]$ is guaranteed by \eqref{eq:grain}. This completes the proof with $C:=1$ and $D:=S^{p'}\mathbb{E}\big[Z^{sp'}\big]$. 
\end{proof}

\section{Examples of Interacting Particle Systems}\label{sec:ipsex}
In each of the upcoming subsections of this section, we will give examples of interacting particle systems on any almost surely locally finite random graph $\mathcal{V}$ which satisfies Condition $(i)$ of Theorem \ref{thm:randips}, i.e., \eqref{eq:aeg}, and \eqref{eq:pm}, e.g., all random graphs with almost surely uniformly bounded degrees, or the random graph models described in the Section \ref{sec:rgex}. In each example of interacting particle system we will define a family of jump rate kernels on the graph, such that for each vertex $v$ of the random graph, $\boldsymbol{c}_v\leq K\deg(v)$ is satisfied for some $K>0$. This and \eqref{eq:pm} readily implies Condition $(ii)$ of Theorem \ref{thm:randips}, i.e., \eqref{eq:cvprod}, and hence Theorem \ref{thm:randips} applies to the example, guaranteeing the existence of a family of interacting particle systems on $\mathcal{V}$ for almost every realization of the model.

\subsection{Divisible Sandpile Model}
$\mathcal{V}$ is an almost surely locally finite random graph with vertex set $V$, satisfying Condition $(i)$ of Theorem \ref{thm:randips}, i.e., \eqref{eq:aeg}, and \eqref{eq:pm} for some $p>1$. We construct a continuous-time sandpile model on $\mathcal{V}$, inspired by the Abelian sandpile model discussed in \cite{Dhar_1999}. More about sandpile models can be found in \cite{corry2018divisors,kalinin2024sandpile}. In our model, every vertex in $V$ has been assigned a non-negative mass of a sandpile standing on it. Each vertex has a cut-off/capacity, and if the mass exceeds the cut-off, the sandpile at that vertex topples and gets distributed equally among its neighbours. The topplings, however, do not occur immediately upon exceeding the cut-off, but at a rate proportional to the degrees of the vertices as well a saturating function of the heights of the sandpiles.

The local state space is $\mathbb{Y} = [0, \infty)$, the non-negative half-line, which is Polish. In our context, the local state represents mass. The entire sandpile at any vertex $v\in V$ ``topples" when its mass $x(v)$ exceeds its ``capacity", some $C_v>0$. The toppled mass is distributed equally among its neighbours. 
If for any realization $\omega$ we denote the new global state after a toppling at any $v\in V$ as $x^{v,\omega}$, then,
\begin{equation*}
 x^{v,\omega}(w)  : = \left\{
\begin{array}{ll}
       0&\text{if }w=v\\
       x(w) + \frac{x(v)}{|\mathcal{N}_{\mathcal{V}(\omega),v}|-1}&\text{if }w \in \mathcal{N}_{\mathcal{V}(\omega),v} \setminus \{v\} \\
       x(w)&\text{if }w \not\in \mathcal{N}_{\mathcal{V}(\omega),v}\\
\end{array} 
\right. 
\end{equation*}

Then, for some $K>0$, the jump rate kernels on $\mathcal{V}$ are then given by,
\begin{equation}\label{eq:jrksandpile}    
\alpha_v(x,\mathcal{A};\omega): = \left\{
\begin{array}{ll}
       0&\text{if }x(v)\leq C_v\\
       K\frac{x(v)}{1+x(v)}\cdot\deg_{\mathcal{V}(\omega)}(v)\mathbbm{1}_{\mathcal{A}}(x^{v,\omega})&\text{if }x(v)>C_v\\
\end{array} 
\right. 
\end{equation}
Measurability of the map $\omega\mapsto\int_{\mathbb{Y}^V}f(y)\alpha_v(x,dy;\omega)$ for every $v\in V$, $x\in\mathbb{Y}^V$ and $f\in\mathcal{C}$ is straighforward (we can write $\deg_{\mathcal{V}(\omega)(v)}=\sum_{w\in V}\mathbbm{1}_{w\in\mathcal{N}_v}(\omega)$ as a sum of indicators). Furthermore, the total jump rate map is given as $\alpha_v(x;\omega)=K\frac{x(v)}{1+x(v)}\deg_{\mathcal{V}(\omega)}(v)\mathbbm{1}_{x(v)>C_v}$, which is LSC, in $x$ for every $v\in V$ and $\omega\in\Omega$.

By the discussion on jump rate kernels on random graphs in Subsection \ref{subs:randfr}, this implies that for each $v\in V$, $\boldsymbol{c}_v$ is a well-defined $[0,\infty]$-valued random variable, and in fact, we have, for each $v\in V$, $\boldsymbol{c}_v=\sup_{x\in\mathbb{Y}^V} \boldsymbol{\alpha}_v(x) = K\deg(v)$. Together with \eqref{eq:pm}, this implies that the model satisfies Condition $(ii)$ of Theorem \ref{thm:randips}, i.e., \eqref{eq:cvprod}, and hence satisfies both the conditions of Theorem \ref{thm:randips}. This enables us to apply Theorem \ref{thm:randips} to the model, and hence we conclude that almost every realization of the model admits a family of interacting particle systems conforming to the corresponding realized jump rate kernels.

\subsection{Consensus Formation Model}
$\mathcal{V}$ is an almost surely locally finite random graph with vertex set $V$, satisfying Condition $(i)$ of Theorem \ref{thm:randips}, i.e., \eqref{eq:aeg}, and \eqref{eq:pm} for some $p>1$. We construct a consensus formation model on $\mathcal{V}$, inspired by the discrete-time model discussed in \cite{Krause2000}. In our model, at every vertex in $V$ is sitting an expert, who has their own opinion but is open to some extent to revise it when being informed about the opinions of all the other experts in their neighbourhood. The revision is a ``weighted average" of all the neighbouring opinions, where weights denote the general consonance between two experts.

The local state space is $\mathbb{Y}:= [0,1]$, which is compact and hence Polish, and which we identify as a space of opinions. $V$ is assigned a family of non-negative weights $(a_{u,v})_{u,v\in V}$, such that for each $u,v\in V$, $a_{u,v}=a_{v,u}$, and for each $v\in V$, $\sum_{u\in V}a_{u,v}=1$. In practice, the graph has some additional geometric structure (e.g., Euclidean) and there is an interplay among the weights, this geometry (e.g., the weights decay with distance at a suitable rate and the weight of a vertex with itself is usually very high, close to $1$), and the edge probabilities. 
The expert at any vertex $v\in V$ revises their opinion $x(v)$ at a rate that proportional to the total disagreement with their neighbours. If for any realization $\omega$ we denote the new global state after a revision of opinion at any $v\in V$ as $x^{v,\omega}$, then,
\begin{equation*}
 x^{v,\omega}(w) :  = \left\{
\begin{array}{ll}
       \sum_{w\in\mathcal{N}_{\mathcal{V}(\omega),v}}a_{v,w}x(w)&\text{if }w=v\\
       x(w)&\text{if }w\neq v\\
\end{array} 
\right. 
\end{equation*}
The jump rate kernels on $\mathcal{V}$ are then given by,
\begin{equation}\label{eq:jrkconsensus}    
\alpha_v(x,\mathcal{A};\omega): = \sum_{w \in \mathcal{N}_{\mathcal{V}(\omega),v}} |x(v)-x(w)|\cdot\mathbbm{1}_\mathcal{A}(x^{v,\omega}).
\end{equation}
Measurability of the map $\omega\mapsto\int_{\mathbb{Y}^V}f(y)\alpha_v(x,dy;\omega)$ for every $v\in V$, $x\in\mathbb{Y}^V$ and $f\in\mathcal{C}$ is straightforward. Furthermore, the total jump rate is $\alpha_v(x;\omega)=\sum_{w \in \mathcal{N}_{\mathcal{V}(\omega),v}} |x(v)-x(w)|$, which is continuous, and hence LSC, in $x$ for every $v\in V$ and $\omega\in\Omega$.

By the discussion on jump rate kernels on random graphs in Subsection \ref{subs:randfr}, this implies that for each $v\in V$, $\boldsymbol{c}_v$ is a well-defined $[0,\infty]$-valued random variable, and in fact, we have, for each $v\in V$, $\boldsymbol{c}_v=\sup_{x\in\mathbb{Y}^V} \boldsymbol{\alpha}_v(x) = \deg(v)$. Together with \eqref{eq:pm}, this implies that the model satisfies Condition $(ii)$ of Theorem \ref{thm:randips}, i.e., \eqref{eq:cvprod}, and hence satisfies both the conditions of Theorem \ref{thm:randips}. This enables us to apply Theorem \ref{thm:randips} to the model, and hence we conclude that almost every realization of the model admits a family of interacting particle systems conforming to the corresponding realized jump rate kernels.

\subsection{Contact Process Based On Random Selection}
$\mathcal{V}$ is an almost surely locally finite random graph with vertex set $V$, satisfying Condition $(i)$ of Theorem \ref{thm:randips}, i.e., \eqref{eq:aeg}, and \eqref{eq:pm} for some $p>1$. We construct a contact process on $\mathcal{V}$ based on random selection. This is a variant of the popular contact process commonplace in the literature on interacting particle systems, e.g., \cite{swart2017course}. In our model, at every vertex in $V$ is an organism which can be either healthy or infected. An infected vertex either chooses any of its healthy neighbours uniformly at random and infects it, or recovers itself to become healthy, without affecting any of its neighbours.

The local state space is $\mathbb{Y}:=\{0,1\}$, where $0$ and $1$ represent a healthy or an infected site respectively. 
An infected organism at any vertex $v\in V$ either infects one of its neighbours chosen uniformly randomly at the rate $\lambda>0$, or recovers itself at the rate $1$. Let for any realization $\omega\in\Omega$, any vertex $v\in V$ and any global configuration $x\in\{0,1\}^V$, $V_0^\omega(v,x):=\big\{w\in\mathcal{N}_{\mathcal{V}(\omega),v}\setminus\{v\}\big|\,x(w)=0\big\}$ denote the set of neighbours of $v$ that are healthy in the current configuration $x$ in the realization $\omega$. If we denote the new global state after an infection by the organism at any $v\in V$ of its neighbouring organism at $w\in\mathcal{N}_{\mathcal{V}(\omega),v}$ as $x^{v,\omega}_w$, and that after a recovery of the organism at any $v\in V$ as $x^{v,\omega}$, then,
\noindent\begin{minipage}{0.5\linewidth}
        \begin{equation*}
 x^{v,\omega}_w(u)  : = \left\{
\begin{array}{ll}
       1&\text{if }u=w\\
       x(u)&\text{if }u\neq w\\
\end{array} 
\right. 
\end{equation*}
    \end{minipage}%
    \begin{minipage}{0.5\linewidth}
        \begin{equation*}
 x^{v,\omega}(u)  : = \left\{
\begin{array}{ll}
       0&\text{if }u=v\\
       x(u)&\text{if }u\neq v\\
\end{array} 
\right. 
\end{equation*}
    \end{minipage}%

The jump rate kernels on $\mathcal{V}$ are then given by,
\begin{equation}\label{eq:jrkcontact}    
\alpha_v(x,\mathcal{A};\omega): = \left\{
\begin{array}{ll}
       0&\text{if }x(v)=0\\
      \lambda\sum_{w\in V_0^\omega(v,x)}\mathbbm{1}_{\mathcal{A}}(x_w^{v,\omega})+\mathbbm{1}_{\mathcal{A}}(x^{v,\omega})&\text{if }x(v)=1\\
\end{array} 
\right.
\end{equation}
Measurability of the map $\omega\mapsto\int_{\mathbb{Y}^V}f(y)\alpha_v(x,dy;\omega)$ for every $v\in V$, $x\in\mathbb{Y}^V$ and $f\in\mathcal{C}$ is straightforward. Furthermore, the total jump rate $\alpha_v(x;\omega)=\mathbbm{1}_{x(v)=1}\big(\lambda\sum_{w\in\mathcal{N}_{\mathcal{V}(\omega),v}}\mathbbm{1}_{x(w)=0}+1\big)$ is LSC in $x$ for every $v\in V$ and $\omega\in\Omega$.

By the discussion on jump rate kernels on random graphs in Subsection \ref{subs:randfr}, this implies that for each $v\in V$, $\boldsymbol{c}_v$ is a well-defined $[0,\infty]$-valued random variable, and in fact, we have, for each $v\in V$, $\boldsymbol{c}_v=\sup_{x\in\mathbb{Y}^V} \alpha_v(x) = \lambda\deg(v)+1\leq(\lambda+1)\deg(v)$. Together with \eqref{eq:pm}, this implies that the model satisfies Condition $(ii)$ of Theorem \ref{thm:randips}, i.e., \eqref{eq:cvprod}, and hence satisfies both the conditions of Theorem \ref{thm:randips}. This enables us to apply Theorem \ref{thm:randips} to the model, and hence we conclude that almost every realization of the model admits a family of interacting particle systems conforming to the corresponding realized jump rate kernels.

\subsection{Two-Coloured Interacting Urn Model}
This is a continuous-time version of a model of recent interest \cite{Kaur_Sahasrabudhe_2023}. $\mathcal{V}$ is an almost surely locally finite random graph with vertex set $V$, satisfying Condition $(i)$ of Theorem \ref{thm:randips}, i.e., \eqref{eq:aeg}, and \eqref{eq:pm} for some $p>1$. At each vertex of $V$ is an urn with some black and some white balls. The local states are the numbers of white and black balls in each urn. Every urn chooses a ball at random from itself, at a rate proportional to its degree, and then, for some fixed $\alpha,\beta,m\in\mathbb{N}$ with $\alpha,\beta\leq m$,
\begin{enumerate}
    \item[(a)] Adds $\alpha$ white balls and $m-\alpha$ black balls to each neighbour, if the chosen ball is white.
    \item[(b)] Adds $m-\beta$ white balls and $\beta$ black balls to each neighbour, if the chosen ball is black.
\end{enumerate}

The local state space is $\mathbb{Y}:=\mathbb{N}_0^2$, which is Polish, and the local states are therefore ordered tuples representing the numbers of white and black balls respectively. 
Let for any $x\in(\mathbb{N}_0^2)^V$ and any $v\in V$, $x_W(v)$ and $x_B(v)$ denote the numbers of white and black balls in the urn at $v$ respectively, i.e., $x(v)=(x_W(v),x_B(v))$. If for any realization $\omega$ we denote the new global state after drawing a white and a black ball from the urn at any $v\in V$ as $x_1^{v,\omega}$ and $x_2^{v,\omega}$ respectively, then,

\begin{equation*}
 x^{v,\omega}_1(w)   = \left\{
\begin{array}{ll}
       x(v)&\text{if }w\not\in\mathcal{N}_{\mathcal{V}(\omega),v}\setminus\{v\}\\
       \big(x_W(w)+\alpha,x_B(w)+m-\alpha\big)&\text{if }w\in\mathcal{N}_{\mathcal{V}(\omega),v}\setminus\{v\}\\
\end{array} 
\right. 
\end{equation*}
\begin{equation*}
 x^{v,\omega}_2(w)   = \left\{
\begin{array}{ll}
       x(v)&\text{if }w\not\in\mathcal{N}_{\mathcal{V}(\omega),v}\setminus\{v\}\\
       \big(x_W(w)+m-\beta,x_B(w)+\beta\big)&\text{if }w\in\mathcal{N}_{\mathcal{V}(\omega),v}\setminus\{v\}\\
\end{array} 
\right. 
\end{equation*}
For some $K'>0$, the random jump rate kernels are then given by,
\begin{equation}\label{eq:jrkcolorurn}    
\alpha_v(x,\mathcal{A};\omega) = \frac{K'\deg_{\mathcal{V}(\omega)}(v)}{1+x_W(v)+x_B(v)} \cdot \left[ x_W(v) \mathbbm{1}_{\mathcal{A}}\big(x_1^{v,\omega}\big) + x_B(v) \mathbbm{1}_{\mathcal{A}}\big(x_2^{v,\omega}\big)\right],
\end{equation}
Clearly, these are random transition kernels on $\mathbb{Y}^{\mathcal{N}_v}$. Measurability of the map $\omega\mapsto\int_{\mathbb{Y}^V}f(y)\alpha_v(x,dy;\omega)$ for every $v\in V$, $x\in\mathbb{Y}^V$ and $f\in\mathcal{C}$ is straightforward. Furthermore, the total jump rate $\alpha_v(x;\omega)=\frac{K'\deg_{\mathcal{V}(\omega)}(v)\big(x_W(v)+x_B(v)\big)}{1+x_W(v)+x_B(v)}$ is continuous and hence LSC in $x$ for every $v\in V$ and $\omega\in\Omega$.

By the discussion on jump rate kernels on random graphs in Subsection \ref{subs:randfr}, this implies that for each $v\in V$, $\boldsymbol{c}_v$ is a well-defined $[0,\infty]$-valued random variable, and in fact, we have, for each $v\in V$, $\boldsymbol{c}_v=\sup_{x\in\mathbb{Y}^V} \alpha_v(x) = \sup_{(W,B)\in\mathbb{N}_0^2}K'\deg(v)\frac{W+B}{1+W+B}=\deg(v)$. Together with \eqref{eq:pm}, this implies that the model satisfies Condition $(ii)$ of Theorem \ref{thm:randips}, i.e., \eqref{eq:cvprod}, and hence satisfies both the conditions of Theorem \ref{thm:randips}. This enables us to apply Theorem \ref{thm:randips} to the model, and hence we conclude that almost every realization of the model admits a family of interacting particle systems conforming to the corresponding realized jump rate kernels.

\subsection{Bak-Sneppen-Based Evolution Model}

This is a continuous-time and infinite-population variant of the Bak-Sneppen evolution model introduced in \cite{baksneppen1993}. $\mathcal{V}$ is an almost surely locally finite random graph with vertex set $V$, satisfying Condition $(i)$ of Theorem \ref{thm:randips}, i.e., \eqref{eq:aeg}, and \eqref{eq:pm} for some $p>1$. At each vertex of $V$ is a species of organism and each species is assigned a fitness level, a number between $0$ and $1$. If the fitness level of the species at any vertex $v\in V$ is the lowest in its neighbourhood, then each species in its neighbourhood, including itself, is assigned an independent fitness level chosen uniformly in $[0,1]$, each of the new fitness levels being independent of all the current ones, and this assignment occurs at a rate proportional to the degree of $v$.

The local state space is $\mathbb{Y}:=[0,1]$, which is compact and therefore Polish, and the local states represent fitness levels of the species. 
For some $K_1>0$, the jump rate kernels on $\mathcal{V}$ are given by,

\begin{equation}\label{eq:jrkbaksnappen}    
\alpha_v(x,\mathcal{A};\omega) = \left\{
\begin{array}{ll}
       0&\text{if }x(v)\neq\min_{w\in\mathcal{N}_{\mathcal{V}(\omega),v}}x(w)\\
       K_1 \sum_{W<V}\mathbbm{1}_{\mathcal{N}_v=W}(\omega)\cdot(|W|-1) \cdot U_W(\mathcal{A})&\text{if }x(v)=\min_{w\in\mathcal{N}_{\mathcal{V}(\omega),v}}x(w)\\
\end{array} 
\right.
\end{equation}
where $U_W$ is the product measure of i.i.d. uniform measures on $[0,1]$ for the coordinates in $W$ and the Dirac measure $\delta_{x(v)}(\cdot)$ for every $v\not\in W$. Clearly, these are random transition kernels on $\mathbb{Y}^{\mathcal{N}_v}$. Measurability of the map $\omega\mapsto\int_{\mathbb{Y}^V}f(y)\alpha_v(x,dy;\omega)$ for every $v\in V$, $x\in\mathbb{Y}^V$ and $f\in\mathcal{C}$ is straightforward. Furthermore the total jump rate $\alpha_v(x;\omega)=\mathbbm{1}_{x(v)=\min_{w\in\mathcal{N}_{\mathcal{V}(\omega),v}}x(w)}\cdot K_1\deg_{\mathcal{V}(\omega)}(v)$ is LSC in $x$ for every $v\in V$ and $\omega\in\Omega$. By the discussion on jump rate kernels on random graphs in Subsection \ref{subs:randfr}, this implies that for each $v\in V$, $\boldsymbol{c}_v$ is a well-defined $[0,\infty]$-valued random variable, and in fact, we have, for each $v\in V$, $\boldsymbol{c}_v=K_1\deg(v)$. Together with \eqref{eq:pm}, this implies that the model satisfies Condition $(ii)$ of Theorem \ref{thm:randips}, i.e., \eqref{eq:cvprod}, and hence satisfies both the conditions of Theorem \ref{thm:randips}. This enables us to apply Theorem \ref{thm:randips} to the model, and hence we conclude that almost every realization of the model admits a family of interacting particle systems conforming to the corresponding realized jump rate kernels.

\section{Existence and Convergence Using Graphical Construction}\label{sec:graph}
The current section discusses graphical construction as a tool to prove existence and convergence of interacting particle systems. In Subsection \ref{subs:graphical}, we will formally introduce the notions of graphical construction, and will place it in the context of this paper through two propositions -- Propositions \ref{prop:exstcggc} and \ref{prop:ggcexstc}. In brief, Proposition \ref{prop:exstcggc} states that the assumptions of Theorem \ref{thm:exist_cvg_ips} are sufficient for the graphical construction on $\mathcal{V}$ to not percolate, and Proposition \ref{prop:ggcexstc} states that if the graphical construction on $\mathcal{V}$ does not percolate, the conclusions of Theorem \ref{thm:exist_cvg_ips} hold. In effect, Propositions \ref{prop:exstcggc} and \ref{prop:ggcexstc} together serve as a proof of Theorem \ref{thm:exist_cvg_ips}. We will prove Propositions \ref{prop:exstcggc} and \ref{prop:ggcexstc} in Subsections \ref{subs:exstcggc} and \ref{subs:ggcexstc} respectively.
 
We have already discussed in Subsection \ref{subs:gen} that proving the existence of interacting particle systems on infinite graphs is non-trivial. To address this issue, we resort to \emph{graphical construction}\index{graphical construction}, a popular technique also referred to in the literature as \emph{Poisson graphical construction}, \emph{graphical representation}, \emph{Poisson representation}, etc. The idea is to show that, in any finite time \cite{Penrose2008existence} or in a small enough time \cite{Durrett1995}, the graphical construction does not percolate, i.e., the set of all vertices that can \emph{affect} or \emph{be affected by} (in a sense that will be formalized in the following subsections) any given vertex in the graph is almost surely finite. Once this is done, the process can then be shown to bootstrap through small increments of time \cite{Harris1972,Durrett1995} or along \emph{generations} \cite{Penrose2008existence}.

\subsection{Graphical Construction}\label{subs:graphical}

We recall our framework discussed in Subsection \ref{subs:gen}. $\mathcal{V}=(V,E)$ is a locally finite graph (we recall that by ``graph" we always refer to a countable, undirected, simple graph with no loops), $\mathbb{Y}$ is a Polish space. We have a family of jump rate kernels $(\alpha_v^*)_{v\in V}$, and $c_v$'s are defined as in \eqref{eq:jrknbd}.
 
We will now introduce a graphical construction on $\mathcal{V}$. Let $(\mathcal{P}_v)_{v\in V}$ be a family of independent homogeneous Poisson point processes in $[0,\infty)\times[0,1]$ of intensities $c_v$ respectively. We label the points of $\mathcal{P}_v$ as $\big(T_i(v),U_i(v)\big)_{i\in\mathbb{N}}$, with $T_1(v)<T_2(v)<T_3(v)<\cdots$ almost surely. In particular, for each $v\in V$, $\big(T_i(v)\big)_{i\in\mathbb{N}}$ is a homogeneous Poisson process on $[0,\infty)$ with intensity $c_v$ and with i.i.d. marks $\big(U_i(v)\big)_{i\in\mathbb{N}}$, such that for each $i\in\mathbb{N}$, $U_i(v)$ is uniformly distributed on $[0,1]$. We define $\mathcal{X}:=\bigcup_{v\in V}\big(\big\{(v,0)\}\cup\{(v,T_i(v))|i\in\mathbb{N}\big\}\big)$, and define an oriented graph $(\mathcal{X},\mathcal{E})$ by putting an oriented edge $(u,T)\to(v,T')$ whenever $v\in\mathcal{N}_u^+$ and $T<T'$. Now we make rigorous the notion of \emph{affecting} that we discussed in the introduction of this section.
\begin{definition}
For $w,v\in V$, we say $w$ \emph{affects} $v$ \emph{before time} $t>0$ if there exists a directed path in $(\mathcal{X},\mathcal{E})$\label{nom:etvw} from $(w,0)$ to $(v,T)$ with $T\leq t$. $E_t(w,v)$ denotes the event that $w$ affects $v$ before time $t$.
\end{definition}

\begin{definition}
\label{a:good_graphical}
\emph{(Graphical Construction)} We call the oriented graph $(\mathcal{X},\mathcal{E})$ (or simply, $\mathcal{X}$) the {\em graphical construction on} $\mathcal{V}$ {\em w.r.t. the jump rate kernels} $(\alpha_v^*)_{v\in V}$. We say that the graphical construction {\em does not percolate}\index{graphical construction!good} if for every $v\in V$ and $t>0$, $v$ affects and is affected by only finitely many $w\in V$ before time $t$, almost surely.
\end{definition}

\begin{remark}\label{rem:sugc}
\hfill
\begin{enumerate}
    \item The rationale behind taking $2$-neighbours while putting oriented edges in $\mathcal{X}$ is that we allow {updates} at a given site to affect the state of their neighbouring sites, and not only their own states. This means, whenever there is a tick at any site $v$, it may potentially update the state of the entire neighbourhood of $v$, and therefore, has an effect on the $2$-neighbours of $v$ at any of their subsequent ticks. It immediately follows that in the case of self-updating systems, for the graphical construction, it suffices to put an oriented edge $(u,T)\to(v,T')$ whenever $v\in\mathcal{N}_u$ and $T<T'$. We call this {the} \emph{$1$-step graphical construction} on $\mathcal{V}$ w.r.t. $(\alpha_v^*)_{v\in V}$. 
    \item We note that in \cite{Penrose2008existence}, all the Poisson processes are i.i.d. with rate $c_{\max}$, where $c_{\max}:=\sup_{v\in V}c_v<\infty$. But in our case $c_{\max}$ may not be finite. This leads us refine Penrose's graphical construction by setting, for each $v\in V$, the rate of the Poisson process $\mathcal{P}_v$ as $c_v$. 
\end{enumerate}
\end{remark}

We will now state Propositions \ref{prop:exstcggc} and \ref{prop:ggcexstc}. We will prove them in the upcoming subsections. We first recall $\Theta_{v}$ and $\Theta_{v}^*$ from Definition \ref{def:djrt}.
\begin{prop}
\label{prop:exstcggc}
Let $\mathcal{V}=(V,E)$ be a locally finite graph and $\mathbb{Y}$ be a Polish space. Let $(\alpha_v^*)_{v\in V}$ be a family of jump rate kernels on $\mathcal{V}$. We assume one of the following holds:
\begin{enumerate}
    \item[(i)]  for every $v\in V$, $\Theta_{v}^*<\infty$,
    \item[(ii)]  for every $v\in V$, the jump rate kernels are are self-updating as in \eqref{con:self_upd} and $\Theta_{v}<\infty$.
\end{enumerate}
Then the graphical construction (resp., $1$-step graphical construction) on  $\mathcal{V}$ w.r.t. $(\alpha_v^*)_{v\in V}$ does not percolate.
\end{prop}
\begin{prop}
\label{prop:ggcexstc}
Let $\mathcal{V}=(V,E)$ be a locally finite graph, $\mathbb{Y}$ be a Polish space, and $(\alpha_v^*)_{v\in V}$ be a family of jump rate kernels on $\mathcal{V}$ satisfying \eqref{eq:jrknbd} (resp., \eqref{con:self_upd} in addition to \eqref{eq:jrknbd}), and $G:=G_V$ be the operator defined at \eqref{eq:gendef}. Let the graphical construction (resp., $1$-step graphical construction) on $\mathcal{V}$ w.r.t. $(\alpha_v^*)_{v\in V}$ not percolate. Then there exists a Markovian family of transition distributions $(\mu_t)_{t\geq0}$ on $\bBY^V$, such that the associated semigroup $(P_t)_{t\geq0}$ has $G$ as the generator on $\mathcal{C}$. Moreover, for any sequence $(W_m)_{m\geq1}$ of finite subsets of $V$ {such that $\liminf_{m\to\infty}W_m=V$}, and any $t\geq0$, \eqref{eq:convsg} holds. There also exists a filtration $(\mathcal{G}_t)_{t\geq0}$ and a Markov family of processes $(\xi^x)_{x\in \bBY^V}$, adapted to $(\mathcal{G}_t)_{t\geq0}$,with transition semigroup $(P_t)_{t\geq0}$.
\end{prop}
The proof of Proposition \ref{prop:exstcggc} is based on the proof of \cite[Lemma $5.1$]{penrose2002monotone} with some non-trivial changes made in order to adapt the proof to our assumptions. This is because, the proof relies on controlling the bounds on $\mathbb{P}\big(E_t(v,w)\big)$ for a suitable $t>0$ and for any $v,w\in V$, and for the general case, this entails looking at arbitrary paths in $\mathcal{V}^2$, which can make computations more involved. The changes we make effectively reduce this to considering only the SAWs in $\mathcal{V}$ and not paths in $\mathcal{V}^2$.

We give a constructive proof of Proposition \ref{prop:ggcexstc}, i.e., we explicitly construct a family of Markov processes on the graph using the jump rate kernels. The proof is largely adapted from \cite[Section $6$]{Penrose2008existence}. However, since we do not necessarily have a uniform bound on the jump rates, our graphical construction is slightly more refined than that in \cite{Penrose2008existence} (see Item $(2)$ of Remark \ref{rem:sugc}) and this necessitates some subtle modifications in the proof which will be remarked throughout the proof as and when needed. These changes, however, are simply technical, and less conceptual than those associated with Proposition \ref{prop:exstcggc}. We give detailed proofs of the results involved, for self-containment. 
\subsection{Proof of Proposition \ref{prop:exstcggc}}\label{subs:exstcggc} We prove Proposition \ref{prop:exstcggc} for the general case, i.e., under its Assumption $(i)$, and remark the changes that need to be made in the proof for the self-updating case. As already mentioned, the proof relies on controlling $\mathbb{P}\big(E_t(v,w)\big)$ for a suitable $t>0$ and for any $v,w\in V$, and for the general case, this can make computations more involved. Therefore, we first try to control the probability of the following more tractable event in Lemma \ref{lem:tail}, which will be crucial in proving Proposition \ref{prop:exstcggc}.
\begin{definition}
For $w,v\in V$, we say $w$ \emph{directly affects} $v$ \emph{before time} $t$ if there exists a directed path $(w,0),(v_1,t_1),(v_2,t_2),\cdots,(v,T)$ in $(\mathcal{X},\mathcal{E})$, such that $T\leq t$, and $\gamma=w,v_1,v_2,\cdots,v\in\text{SAW}_{\mathcal{V}}^*(v)$, i.e., $\gamma$ is a SAW in $\mathcal{V}^2$ and is the remnant of a SAW in $\mathcal{V}$. $E_t'(w,v)$\label{nom:etvw'} denotes this event.
\end{definition}

\begin{lemma}\label{lem:tail}
Let $\mathcal{V}=(V,E)$ be a locally finite graph, $\mathbb{Y}$ be a Polish space, and $(\alpha_v^*)_{v\in V}$ be a family of jump rate kernels on $\mathcal{V}$. Let $0<a<1$. We assume for every $v\in V$, $\Theta_{v}^*<\infty$. Let $\mathcal{X}$ be the graphical construction  on $\mathcal{V}$ w.r.t. $(\alpha_v^*)_{v\in V}$. Then for every $v\in V$ there exist $\delta_v>0$, $\delta_v'>0$, $N_v\in\mathbb{N}$, and $N_v'\in\mathbb{N}$ such that for all $n\geq N_v$ and all $n'\geq N_v'$,
\begin{align}\label{eq:affect}
\mathbb{P}\bigg(\bigcup_{w\in V:\text{dist}_{\mathcal{V}^2}(w,v)\geq n} E_{\delta n}'(v,w)\bigg)\leq(1-a)^{-1}{a^n},\;\text{and}\;\;\mathbb{P}\bigg(\bigcup_{w\in V:\text{dist}_{\mathcal{V}^2}(w,v)\geq n'} E_{\delta'n'}'(w,v)\bigg)\leq c_v(1-a)^{-1}{a^{n'}}.
\end{align}
\end{lemma}
\begin{proof}
As the first step of the proof, we denote,
\begin{equation}\label{eq:ndjrt'}
\overline{\Theta}_{v,(n)}^*:=\Bigg(\sum\limits_{\substack{\gamma:=v,v_1,\cdots,v_n \\ {\gamma\in\text{SAW}_{\mathcal{V},n}^*(v)}}}\prod\limits_{i=1}^nc_{v_i}\Bigg)^\frac{1}{n}\text{ and }\quad\overline{\Theta}_{v}^*:=\limsup\limits_{n\to\infty}\overline{\Theta}_{v,(n)}^*.
\end{equation}
We observe that, for each $n\in\mathbb{N}$, $\big(\overline{\Theta}_{v,(n)}^*\big)^n\leq(\Theta_{v,(n+1)}^*)^n$, and thus taking $\limsup$ on the $n$-th root of both sides, $\overline{\Theta}_{v}^*\leq\Theta_{v}^*$. Thus, for each $v\in V$, $\Theta_{v}^*<\infty$ implies that $\overline{\Theta}_{v}^*<\infty$.

We prove the first inequality of \eqref{eq:affect} first. We fix $\varepsilon>0$. Since $\overline{\Theta}_{v}^*<\infty$ holds for every $v\in V$, there exists $N_{v,\varepsilon}\in\mathbb{N}$ for each $v\in V$, such that for all $m\geq N_{v,\varepsilon}$,
\begin{equation}\label{eq:threshold}
\big(\overline{\Theta}_{v,(m)}^*\big)^m:=\sum\limits_{\substack{\gamma:=v,v_1,\cdots,v_m \\ {\gamma\in\text{SAW}_{\mathcal{V},m}^*(v)}}}\prod\limits_{i=1}^mc_{v_i}\leq(\overline{\Theta}_{v}^*+\varepsilon)^m,
\end{equation}

In order to compute the LHS of the first inequality of \eqref{eq:affect}, we consider all possible paths $\widetilde{\gamma}:=(v,0),(v_1,T_1)$, $(v_2,T_2),\cdots,(v_{{m}-1},T_{{m}-1}),(w,T_m)$ in $\mathcal{P}$ such that $T_m\leq t$ and $\gamma:=v,v_1,v_2,\cdots,w$ is a SAW in $\mathcal{V}^2$ and a remnant of a SAW in $\mathcal{V}$. {However, it is indeed sufficient to look at a smaller set of paths, viz., all possible paths $\widetilde{\gamma}:=(v,0),(v_1,t_1)$, $(v_2,t_2),\cdots,(v_{{m}-1},t_{{m}-1}),(w,T)$ in $\mathcal{P}$ such that $T\leq t$, $\gamma:=v,v_1,v_2,\cdots,w\in\text{SAW}_{\mathcal{V},m}^*(v)$, and for each $1\leq k\leq n$, $t_k$ is the smallest real number greater than $t_{k-1}$ such that $(v_k,t_k)\in\mathcal{P}$. In other words, for each $1\leq k\leq n$, $t_k$ is the time of the first tick at $v_k$ after $t_{k-1}$. This is because if there is an admissible series of ticks within $[0,t]$ along an influencing chain of vertices, then the series of consecutive first ticks along the same chain will be within $[0,t]$.} We write $T=\sum_{i=1}^{{m}}W_{v_i}$ where $W_{v_i}=t_i-t_{i-1}$, for $1\leq i\leq n$ with $t_0:=0$ and $t_{{m}}:=T$. Thus for each $i$, $W_{v_i}$ represents the time upto the first tick at $v_i$ after $t_{i-1}$. From standard properties of homogeneous Poisson processes on the real half line, $W_{v_i}$'s are independent random variables, and for each $1\leq i\leq{m}$, $W_{v_i}$ is exponential with {rate} $c_{v_i}$. Since $T$ {uniquely} corresponds to the path $\gamma$ we write $T$ as $T_\gamma$. Hence for any $\delta>0$ and any $n\in\mathbb{N}$, we have,
\begin{align}
&\mathbb{P}\bigg(\bigcup_{w\in V:\text{dist}_{\mathcal{V}^2}(w,v)\geq n}E_{\delta n}'(v,w)\bigg)
\leq\sum_{m\geq n}\sum_{ \gamma\in\text{SAW}_{\mathcal{V},m}^*(v)}\mathbb{P}\big(T_\gamma\leq\delta n\big)\qquad\text{(by Boole's inequality)}\nonumber\\&
\leq\sum_{m\geq n}\sum_{ \gamma\in\text{SAW}_{\mathcal{V},m}^*(v)}e^{\theta\delta n}\mathbb{E}\big[e^{-\theta T_\gamma}\big]\qquad\text{(using Chernoff bound for any $\theta>0$ to be fixed later)}\nonumber\\&
=\sum_{m\geq n}\sum_{\substack{\gamma=v,v_1,\cdots,v_m \\ \gamma\in\text{SAW}_{\mathcal{V},m}^*(v)}}e^{\theta\delta n}\prod\limits_{i=1}^m\mathbb{E}\big[e^{-\theta W_{v_i}}\big]\;\text{($W_{v_i}$'s are independent)}\nonumber\\&
=\sum_{m\geq n}\sum_{\substack{\gamma=v,v_1,\cdots,v_m \\ \gamma\in\text{SAW}_{\mathcal{V},m}^*(v)}}e^{\theta\delta n}\prod\limits_{i=1}^m\frac{c_{v_i}}{c_{v_i}+\theta}\quad\text{($W_{v_i}$ is exponential with rate $c_{v_i}$)}\nonumber\\&
\leq e^{\theta\delta n}\sum_{m\geq n}\sum_{\substack{\gamma=v,v_1,\cdots,v_m \\ \gamma\in\text{SAW}_{\mathcal{V},m}^*(v)}}\prod\limits_{i=1}^m\frac{c_{v_i}}{\theta}
\leq e^{\theta\delta n}\sum_{m\geq n}\frac{1}{\theta^m}\sum_{\substack{\gamma=v,v_1,\cdots,v_m \\ \gamma\in\text{SAW}_{\mathcal{V},m}^*(v)}}\prod\limits_{i=1}^mc_{v_i}.\label{eq:ineqinpro}
\end{align}
For all $n\geq N_{v,\varepsilon}$, by \eqref{eq:threshold}, \eqref{eq:ineqinpro} implies,
\begin{equation}\label{eq:ineqpenul}
\mathbb{P}\bigg(\bigcup_{w\in V:\text{dist}_{\mathcal{V}^2}(w,v)\geq n}E_{\delta n}'(v,w)\bigg)\leq e^{\theta\delta n}\sum_{m\geq n}\frac{1}{\theta^m}(\overline{\Theta}_{v}^*+\varepsilon)^m
\end{equation}
Now we choose $\theta>0$ such that $\frac{\overline{\Theta}_{v}^*+\varepsilon}{\theta}<a$. Thus $\sum_{m\geq n}\frac{1}{\theta^m}(\overline{\Theta}_{v}^*+\varepsilon)^m<(1-a)^{-1}a^n$. Corresponding to this choice of $\theta$, we choose $\delta_{v,\varepsilon}>0$ such that $e^{\theta\delta_{v,\varepsilon}}<1$. Then, \eqref{eq:ineqpenul} implies, for all $n\geq N_{v,\varepsilon}$,
\begin{equation*}
\mathbb{P}\bigg(\bigcup_{w\in V:\text{dist}_{\mathcal{V}^2}(w,v)\geq n}E_{\delta_{v,\varepsilon}n}'(v,w)\bigg)\leq(1-a)^{-1}a^n.
\end{equation*}
Since the dependence on $\varepsilon$ is intermediary, we can write $N_v:=N_{v,\varepsilon}$ and $\delta_v:=\delta_{v,\varepsilon}$, and this proves the first inequality of \eqref{eq:affect}.

The proof of the second inequality of \eqref{eq:affect} is similar. We observe that to compute the LHS of the second inequality, we have to consider all possible paths $\widetilde{\gamma}':=(w,0),(v_{n-1},t_1)$, $(v_{n-2},t_2)\cdots,(v_1,t_{n-1})$, $(v,T)$ in $\mathcal{P}$ (the indexing is for notational convenience) such that $T\leq t$ and $\gamma:=v,v_1,v_2,\cdots,w$ is a SAW in $\mathcal{V}^2$ and the remnant of a SAW in $\mathcal{V}$. For any $\delta>0$ and any $n\in\mathbb{N}$, we use the same logic used in arriving at \eqref{eq:ineqinpro} to show,
\begin{equation}\label{eq:ineqinpro'}
\mathbb{P}\bigg(\bigcup_{w\in V:\text{dist}_{\mathcal{V}^2}(w,v)\geq n}E_{\delta n}'(w,v)\bigg)\leq e^{\theta\delta n}\sum_{m\geq n}\frac{1}{\theta^m}\sum_{\substack{\gamma=v,v_1,\cdots,v_m \\ \gamma\in\text{SAW}_{\mathcal{V},m}^*(v)}}c_v\prod\limits_{i=1}^{m-1}c_{v_i}.
\end{equation}
We fix $\varepsilon>0$. Since $\Theta_{v}^*<\infty$ holds for every $v\in V$, there exists $N_{v,\varepsilon}'\in\mathbb{N}$ for each $v\in V$, such that for all $m\geq N_{v,\varepsilon}'$,
\begin{equation}\label{eq:threshold2}
\big(\Theta_{v,(m)}^*\big)^{m-1}:=\sum\limits_{\substack{\gamma:=v,v_1,\cdots,v_m \\ {\gamma\in\text{SAW}_{\mathcal{V},m}^*(v)}}}\prod\limits_{i=1}^{m-1}c_{v_i}\leq(\Theta_{v}^*+\varepsilon)^{m-1}.
\end{equation}
For all $n\geq N_{v,\varepsilon}'$, by \eqref{eq:threshold2}, \eqref{eq:ineqinpro'} implies,
\begin{equation}\label{eq:ineqpenul2}
\mathbb{P}\bigg(\bigcup_{w\in V:\text{dist}_{\mathcal{V}^2}(w,v)\geq n}E_{\delta n}'(v,w)\bigg)\leq c_ve^{\theta\delta n}\sum_{m\geq n}\frac{1}{\theta^m}(\Theta_{v}^*+\varepsilon)^{m-1}
\end{equation}
As in the proof of the first inequality of \eqref{eq:affect}, we can find $\delta_v'>0$ such that for all $n\geq N_v':=N_{v,\varepsilon}'$,
\begin{equation*}
\mathbb{P}\bigg(\bigcup_{w\in V:\text{dist}_{\mathcal{V}^2}(w,v)\geq n}E_{\delta_vn}'(v,w)\bigg)\leq c_v(1-a)^{-1}a^n.
\end{equation*}
This proves the second inequality of \eqref{eq:affect}.
\end{proof}
Now we will use Lemma \ref{lem:tail} and the following lemma that asserts that  every path in $\mathcal{V}^2$ has a subpath in $\text{SAW}_{\mathcal{V}}^*$, to prove Proposition \ref{prop:exstcggc}, i.e., to show the graphical construction $\mathcal{X}$ does not percolate.

\begin{lemma}\label{lem:path2saw}
Let $\mathcal{V}$ be any locally finite graph. Let for any $n\in\mathbb{N}$, $\gamma=v,v_1,\cdots,v_n$ be a path in $\mathcal{V}^2$, {such that $v\neq v_n$}. Then there exists a subpath $\gamma'=v,v_{k_1},\cdots,v_{k_l}(=v_n)$ of $\gamma$, such that $\gamma'\in\text{SAW}_\mathcal{V}^*(v)$.
\end{lemma}
\begin{proof}
We do an induction on $n$. For $n=1$, there are two possibilities. Either $v$ and $v_1$ are neighbours in $\mathcal{V}$, or they have a common neighbour $u$. In the first case $v,v_1$ forms a SAW in $\mathcal{V}$, and hence $\gamma'=v,v_1\in\text{SAW}_\mathcal{V}^*(v)$. In the second case $v,u,v_1$ forms a SAW in $\mathcal{V}$ and hence, again, $\gamma'=v,v_1\in\text{SAW}_\mathcal{V}^*(v)$.

Let, for some $m\in\mathbb{N}$ the condition hold for any path of length $m$ in $\mathcal{V}^2$. Now, let $\gamma=v,v_1,\cdots,v_{m+1}$ be a path of length $m+1$ in $\mathcal{V^+}$. Since the path $v,v_1,\cdots,v_m$ is of length $m$, the condition holds for it, and therefore, it admits a subpath $\gamma'=v,v_{k_1},\cdots,v_{k_l}(=v_m)$ which is a SAW in $\mathcal{V}^2$ and a remnant of a SAW $\gamma_1$ in $\mathcal{V}$.

Now, $v_{m+1}$ may or may not appear in $\gamma'$. In the first case, let $v_{m+1}=v_{k_i}$ for some $1\leq i<l$. Then $\gamma'':=v,v_{k_1},\cdots,v_{k_i}(=v_{m+1})$ is the required subpath of $\gamma$. In the second case, let $1\leq j\leq l$ be the smallest index such that $v_{k_j}$ and $v_{m+1}$ are neighbours in $\mathcal{V}^2$. Then either $v_{k_j}$ and $v_{m+1}$ are neighbours in $\mathcal{V}$, or they have a common neighbour in $\mathcal{V}$. In either cases $\gamma''$ obtained by appending $v_{m+1}$ to $v,v_{k_1},\cdots,v_{k_j}$ is the required subpath of $\gamma$. Indeed, let $\gamma_{1,k_j}$ denote the truncation of $\gamma_1$ till $v_{k_j}$. If $v_{k_j}$ and $v_{m+1}$ are neighbours in $\mathcal{V}$, then $\gamma''$ is a remnant of the SAW in $\mathcal{V}$ obtained by appending $v_{m+1}$ to $\gamma_{1,k_j}$. The latter is indeed a SAW as $v_{m+1}$ does not occur in $\gamma_{1,k_j}$, which in turn follows from the definition of $j$. On the other hand, if  $v_{k_j}$ and $v_{m+1}$ have a common neighbour $u_1$ in $\mathcal{V}$, then $\gamma''$ is a remnant of the SAW in $\mathcal{V}$ obtained by appending $u_1,v_{m+1}$ to $\gamma_{1,k_j}$. Again, the latter is indeed a SAW as $u_1,v_{m+1}$ does not occur in $\gamma_{1,k_j}$, which in turn follows from the definition of $j$.
\end{proof}

\begin{proof}[Proof of Proposition \ref{prop:exstcggc}.]

We will first prove the proposition under Assumption $(i)$ and then remark about the changes needed under Assumption $(ii)$.

For any given $v\in V$ and $t>0$, we claim that, almost surely, $v$ directly affects and is directly affected by only finitely many $w\in V$ before time $t$. Let $F_n^v$ be the event that there exists $w\in V$ with $\text{dist}_{\mathcal{V}^2}(w,v)\geq n$, such that $v$ directly affects $w$ before time $t$, i.e., $F_{n,t}^v:=\big\{\cup_{w\in V:\text{dist}_{\mathcal{V}^2}(w,v)\geq n}E_t'(v,w)\big\}$. Then, $F_{n,t}^v$ is a non-increasing sequence of events in $n$, and $\sum_{n\in\mathbb{N}}\mathbb{P}(F_{n}^v)<\infty$ by Lemma \ref{lem:tail}. Also, for any $n\in\mathbb{N}$, by the local finiteness of $\mathcal{V}$, there exist finitely many $w$'s such that $\text{dist}_{\mathcal{V}^2}(w,v)<n$. Thus, using the Borel-Cantelli Lemma, we conclude that almost surely, $v$ directly affects only finitely many vertices before time $t$. The other claim can be proved using similar arguments.

Now we show that $\mathcal{X}$ does not percolate. Let for any $v\in V$ and $t>0$, $\text{Aff}_{v,t}$ and $\overline{\text{Aff}}_{v,t}$ be respectively the set of vertices that are affected and that are directly affected by $v$ in time $t$. Then for each $w\in\text{Aff}_{v,t}$, there exists a path $(v,0),(v_1,t_1),\cdots,(v_{n-1},t_{n-1})$, $(w,T)$ in $(\mathcal{X},\mathcal{E})$, such that $T\leq t$. By Lemma \ref{lem:path2saw}, there exists a subpath $\gamma'=v,v_{k_1},\cdots,v_{k_l}(=w$) of $\gamma=v,v_1,\cdots,w$, for some $l\leq|\gamma|$, which is a SAW in $\mathcal{V}^2$ and a remnant of a SAW in $\mathcal{V}$. Obviously $t_{k_1}<t_{k_2}<\cdots<t_{k_l}=T$, and thus $(v,0),(v_{k_1},t_{k-1}),\cdots,(v_{k_{l-1}},t_{k_{l-1}}),(w,T)$ is a path in $(\mathcal{X},\mathcal{E})$ such that $\gamma'=v,v_{k_1},\cdots,v_{k_l}\in\text{SAW}_{\mathcal{V},n}^*(v)$. Thus $w\in\overline{\text{Aff}}_{v,t}$. {This implies $\text{Aff}_{v,t}\subset\overline{\text{Aff}}_{v,t}$.} By the discussion in the previous paragraph, $\mathbb{P}\big(\big|\overline{\text{Aff}}_{v,t}\big|<\infty\big)=1$. Thus $\mathbb{P}\big(\big|\text{Aff}_{v,t}\big|<\infty\big)=1$, i.e., almost surely, the set of vertices affected by $v$ in time $t$ is finite. Similarly we can show that the set of vertices affecting $v$ in time $t$ is almost surely finite. This proves Proposition \ref{prop:exstcggc} under Assumption $(i)$ of the same.

We now remark about changes to the above proof under Assumption $(ii)$ i.e., for self-updating systems. For such systems, the graphical construction has to be as in Point $(1)$ of Remark \ref{rem:sugc}, and instead of $E_t'(v,w)$, one has to consider $E_t''(v,w)$ for $t>0$, where $E_t''(v,w)$ is the event that there exists a directed path $(v,0),(v_1,t_1),\cdots,(w,T)$ in $\mathcal{X}$ such that $T\leq t$ and $\gamma=v,v_1,\cdots,w\in\text{SAW}_\mathcal{V}(v)$. The rest of the proof is similar to that for general systems.
\end{proof}

\subsection{Proof of Proposition \ref{prop:ggcexstc}}\label{subs:ggcexstc}

Throughout the remainder of this subsection, we work in the premise of Proposition \ref{prop:ggcexstc}, and assume its conditions to hold.

\smallskip
\paragraph*{\emph{Step I: Construction of a partition on $\mathcal{X}$}}
We define a class of subsets $(\mathfrak{G}_k)_{k\geq0}$ of $\mathcal{X}$. For each $k\geq0$. We set $\mathfrak{G}_0:=\{(v,0)|v\in V\}$. For any $k\geq1$, we set,
\begin{equation*}
\mathfrak{G}_k=\big\{(v,T)\in\mathcal{X}|\;\text{the longest directed path from any }(u,0)\in\mathfrak{G}_0\text{ to }(v,T)\text{ is of length }k\big\}.
\end{equation*}

\begin{lemma}\label{lem:partition}
$(\mathfrak{G}_k)_{k\geq0}$ is almost surely a partition of $\mathcal{X}$.
\end{lemma}
\begin{proof}
We observe that the sets $(\mathfrak{G}_k)_{k\geq0}$ are mutually exclusive, since {the length of} a longest directed path, if exists, is unique. Also, let $(v,T)\in\mathcal{X}$ be chosen arbitrarily. Then since $\mathcal{X}$ does not percolate, almost surely only finitely many $w\in V$ can affect $v$ in time $T$, i.e., almost surely, there exist finitely many $(w,0)\in\mathcal{X}$ each with an oriented path to $(v,T)$. We can take a maximum of the lengths of these paths and call it $k$, and thus, almost surely, $(v,T)\in\mathfrak{G}_k$. Thus $(\mathfrak{G}_k)_{k\geq0}$ is almost surely a partition of $\mathcal{X}$.
\end{proof}

\smallskip
\paragraph*{\emph{Step II: Construction of the family of processes on $\mathbb{Y}^V$}}
We now construct our process on $\mathbb{Y}^V$ along $(\mathfrak{G}_k)_{k\geq0}$, inductively. For each $v\in V$ and $x\in\mathbb{Y}^{\mathcal{N}_v}$, we define the probability kernel,
\begin{equation}\label{eq:probker}
\mu_v(x,\cdot):=\frac{\alpha_v^*(x,\cdot)+\big(c_v-\alpha_v^*(x)\big)\delta_x(\cdot)}{c_v}.
\end{equation}
The formulation of \eqref{eq:probker} is motivated by \cite[Equation $(6.9)$]{Penrose2008existence}, with necessary modifications to account for the fact that we no longer have the assumption of uniformly bounded jump rates. By \cite[Lemma $3.22$]{Kallenberg2002prob}, for each $v\in V$, there is a measurable function $\psi_v:\bBY^{\mathcal{N}_v}\times[0,1]\to \bBY^{\mathcal{N}_v}$ such that if $\omega$ is uniformly distributed over $[0, 1]$, then $\psi_v(x,\omega)$ has distribution $\mu_v(x,\cdot)$ for each $x\in \bBY^{\mathcal{N}_v}$.
 
{We will now explicitly construct the process for general systems and will subsequently remark on the changes to be made for self-updating ones.} For each $v\in V$, we set $T_0^*(v)=0$ and list $(T_j(u))_{j\in\mathbb{N},u\in\mathcal{N}_v}$ as $(T_j^*(v))_{j\in\mathbb{N}}$ in increasing order in $j$. Given $x\in \bBY^V$, we construct the process $(\xi_t^x)_{t\geq0}$ with initial value $x$ as follows. For all $v\in V$ we assume that $\xi_t^x(v)$ is right-continuous in $t$ and constant on $\big[T_i^*(v),T_{i+1}^*(v)\big)$ for each $i\in\mathbb{N}$. 
We define the process inductively along $(\mathfrak{G}_k)_{k\geq0}$. We start this by defining the process at all $(v,0)\in\mathfrak{G}_0$, by setting $\xi_0^x(v):=x(v)$ for all $v\in V$. Suppose, for some $k\geq0$, the process is defined for each $(v, T)\in\mathfrak{G}_k$. Let $\big(v,T_j(v)\big)\in\mathfrak{G}_{k+1}$ and  $T:=T_j(v)$. Then, we define the state of the process in $\mathcal{N}_v$ at time $T$ by,
\begin{equation}\label{eq:updaterule}
\xi_T^x|_{\mathcal{N}_v}:=\psi_v\big(\xi_{T-}^x|_{\mathcal{N}_v},{U}_j(v)\big).
\end{equation}
We observe that for each $w\in\mathcal{N}_v$, there exists $i=i_{w,T}$ such that $T:=T_i^*(w)$, and for any $T_j(w)<T$, $(w,T_j(w))$ belongs to the $k$-th generation at the latest. Therefore, $\xi_t^x(w)$ is already defined on $[T_{i-1}^*(w),T)$. Hence $\xi_{T-}^x(w)$ is well-defined.

Now, for self-updating systems, we skip the step which involves re-enumerating the $T$'s as $T^*$'s.  We construct the process $(\xi_t^x)_{t\geq0}$ similarly as in the general case, but in this case $\xi_t^x(v)$ is constant on $\big[T_i(v),T_{i+1}(v)\big)$ for each $i\in\mathbb{N}$. 
We use an inductive argument similar to that in the general case, the update rule being the same as \eqref{eq:updaterule}. The nature of the jump rate kernels (see \eqref{con:self_upd}) force the process at any $v\in V$ to update only at times $\big(T_j(v)\big)_{i\in\mathbb{N}}$. For well-definedness, we observe that for each $w\in\mathcal{N}_v$, there exists $j:=j_{w,T}$ such that $T_j(w)<T<T_{j+1}(w)$, and thus $(w,T_j(w))$ belongs to the $k$-th generation at the latest, so that $\xi_t^x(w)$ is already defined on $[T_j(w),T)$. Hence $\xi_{T-}^x(w)$ is well-defined.
\smallskip
 
\paragraph*{\emph{Step III: Showing Existence}} In the following lemma, which is based on \cite[Lemma $6.3$]{Penrose2008existence}, we show the existence part of Proposition \ref{prop:ggcexstc}, i.e., we show that the family of processes described above is Markovian, is adapted to a certain filtration, and the corresponding transition semigroup has the generator $G$ on $\mathcal{C}$. The proof follows the proof of \cite[Lemma $6.3$]{Penrose2008existence} with minor alterations to account for the possible lack of a uniform bound on the jump rates.
\begin{lemma}\label{lem:mrkv}
We assume the conditions of Proposition \ref{prop:ggcexstc} to hold. Then there exists a filtration $(\mathcal{G}_t)_{t\geq0}$ with respect to which the family of processes $(\xi^x)_{x\in \bBY^V}$ described above is a Markov family of processes on $\bBY^V$. The associated family of transition distributions induces a semigroup of operators $(P_t)_{t\geq0}$ with generator $G$ given on $\mathcal{C}$ by \eqref{eq:gendef}, i.e.,
\begin{equation*}
Gf=\lim\limits_{t\to0}t^{-1}(P_tf-f)
\end{equation*}
holds with uniform convergence for all $f\in\mathcal{C}$.
\end{lemma}
    
\begin{proof}
Let $\mathcal{G}_t$ denote the $\sigma$-field generated by $\big\{{\big(T_i(v),U_i(v)\big)}\big|v\in V,T_i(v)\leq  t\big\}$. By our construction of $(\xi^x)_{x\in\mathbb{Y}^V}$ and standard properties of homogeneous Poisson processes, $(\xi^x)_{x\in \bBY^V}$ are a Markov family of processes adapted to filtration $(\mathcal{G}_t)_{t\geq0}$.
Since $\bBY^V$ has product topology, sample paths are c\`adl\`ag in each copy of $\bBY$, and {since }there are no limit points of $\big(T_i(v)\big)_{i\in\mathbb{N}}$ at any site $v\in V$, sample paths in $\bBY^V$ are c\`adl\`ag.

Let $(P_t)_{t\geq0}$ be the transition semigroup corresponding to $(\xi^x)_{x\in \bBY^V}$. We will now show for general systems that $G$ is the generator of  $(P_t)_{t\geq0}$ on $\mathcal{C}$. For every $f\in\mathcal{C}$ and $x\in \bBY^V$,
\begin{equation}\label{eq:semig_equiv}
P_tf(x)-f(x)=\mathbb{E}\big[f(\xi_t^x)-f(x)\big].
\end{equation}
We choose $A<V$ such that $f\in\mathcal{C}(A)$. Then $f(\xi_t^x)\neq f(x)$ only if $\min\{T_1(v),v\in\mathcal{N}_A\}\leq t$. Accordingly, we consider the events $F_t^{(1)}:=\big\{T_1(v)\leq t\text{ for two or more }v\in\mathcal{N}_A^+\big\}$, $F_t^{(2)}:=\big\{T_2(v)\leq t\text{ for some }v\in\mathcal{N}_A\big\}$, and, for each $v\in\mathcal{N}_A$, $F_{t,v}^{(3)}:=\big\{T_1(v)\leq t < T_2(v)\text{ and }T_1(w)>t\text{ for }w\in\mathcal{N}_A^+\setminus\{v\}\big\}$. Let $F_t:=F_t^{(1)}\cup F_t^{(2)}$. We claim that,
\begin{equation}\label{eq:claimet}
\lim\limits_{t\to0}t^{-1}\mathbb{P}(F_t)=0, 
\end{equation}
and that for any $v\in\mathcal{N}_A$,
\begin{equation}\label{eq:claimetv}
\lim\limits_{t\to0}t^{-1}\mathbb{P}(F_{t,v}^{(3)})=c_v.
\end{equation}
We assume \eqref{eq:claimet} and \eqref{eq:claimetv} are true and complete the proof. Indeed, we observe, {using \eqref{eq:probker},}
\begin{align}\label{eq:gencv}
&\mathbb{E}[f(\xi_t^x)-f(x)|F_{t,v}^{(3)}]=\int_{\bBY^{\mathcal{N}_v}}\big(f(x|v|y)-f(x)\big)\mu_v(x|_{\mathcal{N}_v},dy)\nonumber\\&
=c_v^{-1}\int_{\bBY^{\mathcal{N}_v}}\big(f(x|v|y)-f(x)\big)\alpha_v^*(x|_{\mathcal{N}_v},dy)=c_v^{-1}G_vf(x).\quad\text{(using \eqref{eq:genclosed})}
\end{align}
We note that $F_t\cap F_{t,v}^{(3)}=\emptyset$ for every $v\in\mathcal{N}_A$, and $\big\{f(\xi_t^x)\neq f(x)\big\}\subset F_t\cup\big(\bigcup_{v\in\mathcal{N}_A}F_{t,v}^{(3)}\big)$. Therefore we use successively \eqref{eq:semig_equiv}, the law of total {expectation}, \eqref{eq:gencv}, and \eqref{eq:gen0} to obtain, for every $f\in\mathcal{C}(A)$,
\begin{align}
&\lim\limits_{t\to0}\sup_{x\in\mathbb{Y}^V}\big|t^{-1}\big(P_tf(x)-f(x)\big)-Gf(x)\big|\nonumber\\&
=\lim\limits_{t\to0}\sup_{x\in\mathbb{Y}^V}\bigg|t^{-1}\mathbb{P}(F_t)\mathbb{E}\big[f(\xi_t^x)-f(x)|F_t\big]+\sum\limits_{v\in\mathcal{N}_A}\Big(c_v^{-1}t^{-1}\mathbb{P}\big(F_{t,v}^{(3)}\big)-1\Big)G_vf(x)\bigg|\quad(\text{using }\eqref{eq:gendef})\nonumber\\&
\leq 2||f||\bigg(\lim\limits_{t\to0}t^{-1}\mathbb{P}(F_t) + c_{\mathcal{N}_A}\sum\limits_{v \in \mathcal{N}_A}\big|c_v^{-1}\lim\limits_{t\to0}t^{-1}\mathbb{P}\big(F_{t,v}^{(3)}\big)-1\big|\bigg)\quad\text{(using \eqref{eq:genbd})}\nonumber\\&
=2||f||\bigg(0+  c_{\mathcal{N}_A}\sum \limits_{v \in \mathcal{N}_A}\big(c_v^{-1}c_v-1 \big) \bigg)=0.\quad\text{(using \eqref{eq:claimet} and \eqref{eq:claimetv})}\label{eq:final}
\end{align}
This shows that the generator of the semigroup $(P_t)_{t\geq0}$ on $\mathcal{C}$ is $G$.

Now, we prove \eqref{eq:claimet} and \eqref{eq:claimetv}, and this will complete the proof. We first observe that for any $W<V$, since the Poisson processes $\big\{\big(T_i(v)\big)_{i\in\mathbb{N}}\big\}_{v\in W}$ are independent, their superposition gives us a new Poisson process which we denote as $\big(T_i(W)\big)_{i\in\mathbb{N}}$ with rate $\widehat{c}_W:=\sum_{w\in W}c_w<\infty$ (since $|W|<\infty$). Now, for any $t>0$, $F_t^c=F_{t,(1)}\cup F_{t,(2)}\cup F_{t,(3)}$, where $F_{t,(1)}:=\big\{T_1(v)>t\text{ for all }v\in\mathcal{N}_A^+\big\}$, $F_{t,(2)}:=\big\{T_1(v)\leq t\text{ for exactly one }v\in\mathcal{N}_A^+\setminus\mathcal{N}_A\text{ and }T_1(w)>t\text{ for all }w\in\mathcal{N}_A\big\}$, and $F_{t,(3)}:=\big\{T_1(v)\leq t<T_2(v)\text{ for exactly one }v\in\mathcal{N}_A\text{ and }T_1(w)>t\text{ for all }w\in\mathcal{N}_A^+\setminus\mathcal{N}_A\big\}$, and $F_{t,(1)},\,F_{t,(2)},\,F_{t,(3)}$ are mutually exclusive. Using standard properties of the Poisson process, one can easily verify that, $\mathbb{P}\big(F_{t,(1)}\big)=e^{-t\widehat{c}_{\mathcal{N}_A^+}}$, $\mathbb{P}\big(F_{t,(2)}\big)=\sum_{v\in\mathcal{N}_A^+\setminus\mathcal{N}_A}(1-e^{-c_vt})\cdot e^{-t\widehat{c}_{\mathcal{N}_A^+\setminus\{v\}}}$, and $\mathbb{P}\big(F_{t,(3)}\big)=t\sum_{v\in\mathcal{N}_A}e^{-c_vt}c_v\cdot e^{-t\widehat{c}_{\mathcal{N}_A^+\setminus\{v\}}}$. {Also, for any $t\geq0$, any $v\in\mathcal{N}_A$, $\mathbb{P}\big(F_{t,v}^{(3)}\big)=c_vte^{-tc_v}\cdot e^{-t\widehat{c}_{\mathcal{N}_A^+\setminus\{v\}}}$.} Thus, $\mathbb{P}(F_t)=o(t)$ and $\mathbb{P}(F_{t,v}^{(3)})=c_vt+o(t)$, and hence \eqref{eq:claimet} and \eqref{eq:claimetv} follow. This completes the proof of the lemma for the general case.

For the self-updating case, once again, since for any $f\in\mathcal{C}$, $f\in\mathcal{C}(A)$ for some $A<V$. Then $f(\xi_t^x)\neq f(x)$ only if $\min\{T_1(v),v\in A\}\leq t$. Accordingly, we consider the events $\overline{F}_t^{(1)}:=\big\{T_1(v)\leq t\text{ for two or more }v\in\mathcal{N}_A\big\}$, $\overline{F}_t^{(2)}:=\big\{T_2(v)\leq t\text{ for some }v\in A\big\}$, and, for each $v\in A$, $\overline{F}_{t,v}^{(3)}:=\big\{T_1(v)\leq t\leq T_2(v)\text{ and }T_1(w)>t\text{ for }w\in\mathcal{N}_A\setminus\{v\}\big\}$, and complete the proof along the same lines as above.
\end{proof}

\smallskip
\paragraph*{\emph{Step IV: Showing Convergence -- Construction of the family of ``truncated" processes on $\mathbb{Y}^V$}} Now we will prove the convergence part of Proposition \ref{prop:ggcexstc}, i.e., we will show that the semigroup $(P_t)_{t\geq0}$ satisfies \eqref{eq:convsg}. 

For $W<V$, we recall that $G_W$ is the generator of a jump process with jump rate kernel $\sum_{v\in{W}}\alpha_v$. Intuitively, this is the jump process with the Poisson clocks outside $W$ are all ``switched off", and therefore, with no jumps taking place in $V\setminus\mathcal{N}_W$. We construct such a family of ``truncated" jump processes as follows by setting $\mathcal{X}_W:=\mathcal{X}\cap\big(W\times[0,\infty)\big)$ (we use the same family of Poisson processes $(\mathcal{P}_v)_{v\in V}$ for every $W<V$),  and partitioning $\mathcal{X}_A$ in exactly the same manner as for $\mathcal{X}$ described in Step I (before Lemma \ref{lem:partition}). Then using the updating rule \eqref{eq:updaterule} but now restricting attention to arrival times $\big(T_j(v)\big)_{v\in W}$, we obtain a pure jump process $(\xi_t^{W,x})_{t\geq0}$ with generator $G_W${, and transition semigroup $(P_t^W)_{t\geq0}$}.

\begin{proof}[Proof of Proposition \ref{prop:ggcexstc}.] In view of Lemma \ref{lem:mrkv}, it only remains to be shown that $(P_t)_{t\geq0}$ satisfies \eqref{eq:convsg}. We first discuss the case of general systems. Let $f\in\mathcal{C}$. Then $f\in\mathcal{C}(A)$ for some $A< V$. Now, for every $(v,t)\in V\times[0,\infty)$, we define $C_{v,t}:=\big\{w\in V\big|E_t(w,z)\text{ for some }z\in\mathcal{N}_v^+\big\}$ and call it the $(v,t)$-\emph{cluster}\index{graphical construction!cluster}. It can be proved that (see \cite[Lemma $6.4$]{Penrose2008existence}) for any $(v,t)\in V\times[0,\infty)$, $x\in \bBY^V$, and $W< V$, if $C_{v,t}\subseteq W$, then $\xi_t^{W,x}(w)=\xi_t^x(w)$ for all $w\in\mathcal{N}_v$. Further, since the graphical construction does not percolate, it follows that $\mathbb{P}\big(|C_{v,t}|<\infty\big)=1$ for every $(v,t)\in V\times[0,\infty)$.  This in turn implies that for any sequence $(W_m)_{m\geq1}$ of finite subsets of $V$ with $\liminf\limits_{m\to\infty}W_m=V$
and for all $v\in A$, $\xi_t^{W_m,x}(v)=\xi_t^x(v)$ for large enough $m$, almost surely. This is because, $\cup_{v\in A}C_{v,t}$ is almost surely finite and is thus almost surely contained in $W_m$ for large enough $m$. Hence, $f\big(\xi_t^{W_m,x}\big)=f(\xi_t^x)$ for large enough $m$, almost surely. Hence the dominated convergence theorem gives us \eqref{eq:convsg}, i.e.,
\begin{equation*}
P_tf(x)=\mathbb{E}[f(\xi_t^x)]=\lim\limits_{m\to\infty}\mathbb{E}[f(\xi_t^{W_m,x})]=\lim\limits_{m\to\infty}P_t^{W_m}(x)
\end{equation*}
 
We now remark about the self-updating systems. For such systems, for every $(v,t)\in V\times[0,\infty)$, we define $C_{v,t}:=\{w\in V|E_t(w,z)\text{ for some }z\in\mathcal{N}_v\}$. This implies that for any $(v,t)\in V\times[0,\infty)$, $x\in \bBY^V$, and $A< V$, if $C_{v,t}\subseteq A$, then $\xi_t^{A,x}(v)=\xi_t^x(v)$. Rest of the proof is same as above.
\end{proof}

\section*{Acknowledgments}
\noindent The author would like to thank his PhD supervisor Dr. D. Yogeshwaran for his support and guidance. The author would also like to thank Dr. Parthanil Roy (Indian Statistical Institute, Bangalore Center), Dr. Manjunath Krishnapur (Indian Institute of Science, Bangalore), Dr. Sarvesh R. Iyer (Ashoka University, India), Dr. Bart\l omiej B\l aszczyszyn (Inria Paris Centre), Dr. Ali Khezeli (Inria Paris Centre) and Dr. Maria Eul\'alia-Vares (Federal University of Rio de Janeiro) for their constructive comments that improved the quality of this paper.

\section*{Funding}
\noindent The author gratefully acknowledges funding from Indian Statistical Institute through the Senior Research Fellowship.

\printnomenclature

\bibliographystyle{alpha}
\bibliography{IPS.bib}

\end{document}